\documentclass{article}
\usepackage[T1]{fontenc}
\usepackage{tgpagella}
\usepackage[a4paper,left=1.1in,right=1.1in,top=1.1in,bottom=1.1in,footskip=0.5in]{geometry}

\input{our_macros.sty}
\begin{document}

\title{Spectral Alignment of Correlated Gaussian matrices.}
\author{Luca Ganassali\footnote{
		INRIA, DI/ENS, PSL Research University, Paris, France. Email: \texttt{luca.ganassali@inria.fr}} , Marc Lelarge\footnote{
		INRIA, DI/ENS, PSL Research University, Paris, France. Email: \texttt{marc.lelarge@ens.fr}} , Laurent Massoulié\footnote{
		MSR-INRIA Joint Centre, INRIA, DI/ENS, PSL Research University, Paris, France. Email: \texttt{laurent.massoulie@inria.fr}}}
\date{\today}
\maketitle
\begin{abstract}
	In this paper we analyze a simple spectral method (\emph{EIG1}) for the problem of matrix alignment, consisting in aligning their leading eigenvectors: given two matrices $A$ and $B$, we compute $v_1$ and $v'_1$ two corresponding leading eigenvectors. The algorithm returns the permutation $\hat{\pi}$ such that the rank of coordinate $\hat{\pi}(i)$ in $v_1$ and that of coordinate $i$ in $v'_1$ (up to the sign of $v'_1$) are the same. 	
	
	We consider a model of weighted graphs where the adjacency matrix $A$ belongs to the Gaussian Orthogonal Ensemble (GOE) of size $N \times N$, and $B$ is a noisy version of $A$ where all nodes have been relabeled according to some planted permutation $\pi$, namely $B= \Pi^T (A+\sigma H) \Pi $, where $\Pi$ is the permutation matrix associated with $\pi$ and $H$ is an independent copy of $A$. We show the following zero-one law: with high probability, under the condition $\sigma N^{7/6+\epsilon} \to 0$ for some $\eps>0$, \emph{EIG1} recovers all but a vanishing part of the underlying permutation $\pi$, whereas if $\sigma N^{7/6-\epsilon} \to \infty$, this method cannot recover more than $o(N)$ correct matches.
	
	This result gives an understanding of the simplest and fastest spectral method for matrix alignment (or complete weighted graph alignment), and involves proof methods and techniques which could be of independent interest.
\end{abstract}

\section{Introduction}
\subsubsection*{The graph alignment problem}
Graph alignment (or graph matching, network alignment) consists in recovering the underlying vertex correspondence between two correlated graphs, and hence can be viewed as the noisy version of the isomorphism problem. 
Many questions can be phrased as graph alignment problems. They are found in various fields, such as network privacy and data de-anonymization \cite{Narayanan08,Narayanan09}, biology and protein-protein interaction networks \cite{Singh08}, natural language processing \cite{Haghighi05}, as well as pattern recognition in image processing \cite{CFVS04}.

For two graphs of size $N$ with adjacency matrices $A$ and $B$, the graph matching problem can be formalized as an optimization problem: 
\begin{equation}
\label{QAP}
\argmax_{P \in \mathcal{S}_N} \langle A, P B P^T \rangle ,
\end{equation}
where the maximum is taken over all $N \times N$ permutation matrices, and $\langle \cdot,\cdot \rangle$ is the canonical matrix inner product. Note that for each $P \in \mathcal{S}_N, \; P B P^T $ is the matrix obtained from $B$ when relabeling the nodes according to $P^{-1}$. This formulation is a special case of the well studied \textit{quadratic assignment problem} (QAP) \cite{Pardalos94}, which is known to be NP-hard in the worst case, as well as some of its approximations \cite{Makarychev14}. A natural idea is then to study the average-case version of this problem, when $A$ and $B$ are random instances. Following a recent line of work \cite{Feizi16,Ding18,Fan2019Wigner}, this paper focuses on the case where the signal lies in the weights of edges between all pairs of nodes.

\subsubsection*{Related work} Some general spectral methods for random graph alignment are introduced in \cite{Feizi16}, based on representation matrices and low-rank approximations. These methods are tested over synthetic graphs and real data; however no precise theoretical guarantee -- e.g. an error control of the inferred mapping depending on the signal-to-noise ratio -- can be found for such techniques. 

Most recently, a spectral method for matrix and graph alignment (\emph{GRAMPA}) was proposed in \cite{Fan2019Wigner,fan2019ERC} and computes a similarity matrix which takes into account all pairs of eigenvalues $(\lambda_i, \mu_j)$ and eigenvectors $(u_i,v_j)$ of matrices $A$ and $B$. The authors study the regime in which the method exactly recovers the underlying vertex correspondence, meeting the state-of-the-art performances for alignment of Erdös-Rényi graphs in polynomial time, and improving the performances among spectral methods for matrix alignment. This method can tolerate a noise $\sigma$ up to $O\left(1/\log N\right)$ to recover the entire underlying vertex correspondence. Since the computations of all eigenvectors is required, the time complexity of \emph{GRAMPA} is at least $O(N^3)$. 

It is important to note that the signs of eigenvectors are ambiguous: in order to optimize the cost function in practice, it is necessary to test over all possible signs of eigenvectors. This additional complexity has no consequence when reducing $A$ and $B$ to rank-one matrices, but becomes costly when the reduction made is of rank $k \gg 1$. This combinatorial observation makes implementation and analysis of general rank-reduction methods (as the ones proposed in \cite{Feizi16}) more difficult. We therefore focus on the analysis of the rank-one reduction (\emph{EIG1} hereafter) which is the simplest and most natural spectral alignment method, where only the leading eigenvectors of $A$ and $B$ are computed, with time complexity $O(N^2)$, which is significantly less than \emph{GRAMPA}.

\subsubsection*{Gaussian weighted graph matching: model and method} As mentioned above, we focus on the case where the graphs are complete, weight-correlated. Matrices $A$ and $B$ are thus symmetric, with correlated entries. A natural model recently studied in \cite{Feizi16,Ding18,Fan2019Wigner} is as follows: $A$ and $H$ are two $N \times N$ independent normalized matrices of the Gaussian Orthogonal Ensemble (GOE), i.e. such that for all $1 \leq i \leq j \leq N$, 
\begin{equation}
\label{GOEmodel}
A_{i,j} = A_{j,i} \sim \begin{cases}
\frac{1}{\sqrt{N}} \mathcal{N}(0,1) & \text{if $i \neq j$}, \\
\frac{\sqrt{2}}{\sqrt{N}} \mathcal{N}(0,1) & \text{if $i = j$},
\end{cases}
\end{equation} and $H$ is an independent copy of $A$. We define $B=\Pi^T \left(A+\sigma H\right) \Pi$, where $\Pi$ is the matrix of some permutation $\pi$ -- e.g. random uniform -- of $\left\lbrace 1,\ldots,N\right\rbrace$ (i.e such that $\Pi_{i,j}=1$ iff $i = \pi(j)$), and $\sigma = \sigma(N)$ is the \textit{noise parameter}.

Given two vectors $x=\left(x_1,\ldots,x_n\right)$ and $y=\left(y_1,\ldots,y_n\right)$ having all distinct coordinates, the permutation $\rho$ which \textit{aligns} $x$ and $y$ is the permutation such that for all $1 \leq i \leq n$, the rank (for the usual order) of $x_{\rho(i)}$ in $x$ is the rank of $y_{i}$ in $y$.

\begin{remark}
\label{densitylebesgue}
Note that in our model, all the probability distributions are absolutely continuous with respect to Lebesgue measure, thus the eigenvectors of $A$ and $B$ all have almost surely pairwise distinct coordinates.
\end{remark}

We recall that the aim is to infer the underlying permutation $\Pi$ given the observation of $A$ and $B$. We now introduce our simple spectral algorithm derived from \cite{Feizi16}, which we call \emph{EIG1}, that consists in computing and aligning the leading eigenvectors $v_1$ and $v'_1$ of $A$ and $B$. This very natural method can be thought of as the relaxation of the QAP formulation (\ref{QAP}) when reducing $A$ and $B$ to rank-one matrices $\lambda_1 v_1 v_1^T$ and $\lambda'_1 v'_1 v_1^{'T}$. Indeed, as soon as $v_1$ and $v'_1$ have pairwise distinct coordinates, is it easy to see that
\begin{equation*}
\argmax_{P \in \mathcal{S}_N} \langle \lambda_1 v_1 v_1^T, P \lambda'_1 v'_1 v_1^{'T} P^T \rangle = \argmax_{P \in \mathcal{S}_N} \pm v_1^T P v'_1  = \rho,
\end{equation*} where $\rho$ is the aligning permutation of $v_1$ and $\pm v'_1$. Computing the two normalized leading eigenvectors (i.e. corresponding to the highest eigenvalues) $v_1$ and $v'_1$ of $A$ and $B$, the \emph{EIG1} algorithm returns the aligning permutation of $v_1$ and $\pm  v'_1$. The method then decides which permutation to output according to the scores:

\begin{algorithm}[H]
	\label{algo_EIG1}
	\caption{\emph{EIG1} Algorithm for matrix alignment}
	\SetAlgoLined
	Compute $v_1$ a normalized leading eigenvector of $A$\;
	Compute $v'_1$ a normalized leading eigenvector of $B$\;
	Compute $\Pi_{+}$ the permutation aligning $v_1$ and $v'_1$\;
	Compute $\Pi_{-}$ the permutation aligning $v_1$ and $-v'_1$\;
	\eIf{$\langle A, \Pi_{+} B \Pi_{+}^T \rangle \geq \langle A, \Pi_{-} B \Pi_{-}^T \rangle$}{
		return $\Pi_{+}$}
	{
		return $\Pi_{-}$
	}
\end{algorithm}

The aim of this paper is to find the regime in which \emph{EIG1} achieves \emph{almost exact recovery}, i.e. recovers all but a vanishing fraction of nodes of the planted truth $\Pi$.

\section{Notations, main results and proof scheme}
\label{notations}
In this section we introduce some notations that will be used throughout this paper, we mention the main results and the proof scheme. 

\subsection{Notations} 
\label{notations_def}

\begin{itemize}
\item Recall that $A$ and $H$ are two $N \times N$ matrices drawn under model (\ref{GOEmodel}) here above. The matrix $B$ is defined as $\Pi^T \left(A+\sigma H\right) \Pi $, where $\Pi$ is a uniform $N \times N$ permutation matrix and $\sigma$ is the noise parameter, depending on $N$. 

\item In the following, $\left(v_1, v_2, \ldots, v_N\right)$ (resp. $\left(v'_1, v'_2, \ldots, v'_N\right)$) denote two orthonormal bases of eigenvectors of $A$ (resp. of $B$) with respect to the (real) eigenvalues $\lambda_1~\geq~\lambda_2~\geq~\ldots~\geq~\lambda_N$ of $A$ (resp. $\lambda'_1~\geq~\lambda'_2~\geq~\ldots~\geq~\lambda'_N$ of $B$). Through all the study, the sign of $v'_1$ is fixed such that $\langle \Pi v_1,v'_1 \rangle >0$.

\item Denote by $\| \cdot \|$ the euclidean norm of $\mathbb{R}^N$. Let $\langle \cdot, \cdot \rangle$ denote the corresponding inner product.

\item For any estimator $\hat{\Pi}$ of $\Pi$, define its overlap:
\begin{equation}
\label{overlap}
\mathcal{L}(\hat{\Pi},\Pi) := \frac{1}{N} \sum_{i=1}^{N} \mathbf{1}_{\hat{\Pi}(i)=\Pi(i)}.
\end{equation} This metric is used to quantify the quality of a given estimator of $\Pi$.

\item The equality $\overset{(d)}{=}$ will refer to equality in distribution. Some event $A_N$ is said to hold \textit{with high probability} (we will use the abbreviation "w.h.p."), if $\dP(A_N)$ converges to $1$ when $N \to \infty$.

\item For two random variables $u=u(N)$ and $v=v(N)$, we will use the notation $u=o_{\mathbb{P}}\left(v\right)$ if $\frac{u(N)}{v(N)} \overset{\mathbb{P}}{\longrightarrow} 0$ when $N \to \infty$. We also use this notation when $X=X(N)$ and $Y=Y(N)$ are $N-$dimensional random vectors: $X=o_{\mathbb{P}}\left(Y\right)$ if $\frac{\| X(N) \|}{\| Y(N) \|} \overset{\mathbb{P}}{\longrightarrow} 0$ when $N \to \infty$.

\item Define 
\begin{equation}
\label{nologfunctions}
\mathcal{F} := \left\lbrace f : \mathbb{N} \to \mathbb{R} \; | \; \forall t>0, N^t f(N) \to \infty, \frac{f(N)}{N^t} \to 0 \right\rbrace. 
\end{equation} For two random variables $u=u(N)$ and $v=v(N)$, $u \asymp v$ refers to equivalence with high probability up to some sub-polynomial factor, meaning that there exists a function $f \in \mathcal{F}$ such that
\begin{equation}
\label{defasymp}
\mathbb{P}\left(\frac{v(N)}{f(N)} \leq u(N) \leq f(N) v(N)\right) \to 1.
\end{equation}
\end{itemize}

Throughout the paper, all limits are taken when $N \to \infty$, and the dependency in $N$ will most of the time be eluded, as an abuse of notation.

\subsection{Main results, proof scheme} 
The result shown can be stated as follows: there exists a condition -- a threshold -- on $\sigma$ and $N$ under which the \emph{EIG1} method enables us to recover $\Pi$ almost exactly, in terms of the overlap $\mathcal{L}$ defined in (\ref{overlap}). Above this threshold, we show that \emph{EIG1} Algorithm cannot recover more than a vanishing part of $\Pi$. 

\begin{theorem}[Zero-one law for \emph{EIG1} method]
\label{theorem_01law}
For all $N$, $\Pi_N$ denotes an arbitrary permutation of size $N$, $\hat{\Pi}_N$ is the estimator obtained with Algorithm \emph{EIG1}, for $A$ and $B$ of model (\ref{GOEmodel}), with permutation $\Pi_N$ and noise parameter $\sigma$. We have the following zero-one law:
	\begin{itemize}
		\item[$(i)$] If there exists $\epsilon>0$ such that $\sigma = o(N^{-7/6-\epsilon})$ then $$\mathcal{L}(\hat{\Pi}_N,\Pi_N) \overset{L^1}{\longrightarrow} 1.$$
		\item[$(ii)$] If there exists $\epsilon>0$ such that $\sigma = \omega(N^{-7/6+\epsilon})$ then $$\mathcal{L}(\hat{\Pi}_N,\Pi_N) \overset{L^1}{\longrightarrow} 0.$$
	\end{itemize}
\end{theorem}

Results of Theorem \ref{theorem_01law} are illustrated on Figure \ref{image_overlap} showing the zero-one law at $\sigma \asymp N^{-7/6}$. Note that the convergence to the step function appears to be slow.

\begin{figure}[H]
	\centering
	\includegraphics[width=14cm,height=8cm]{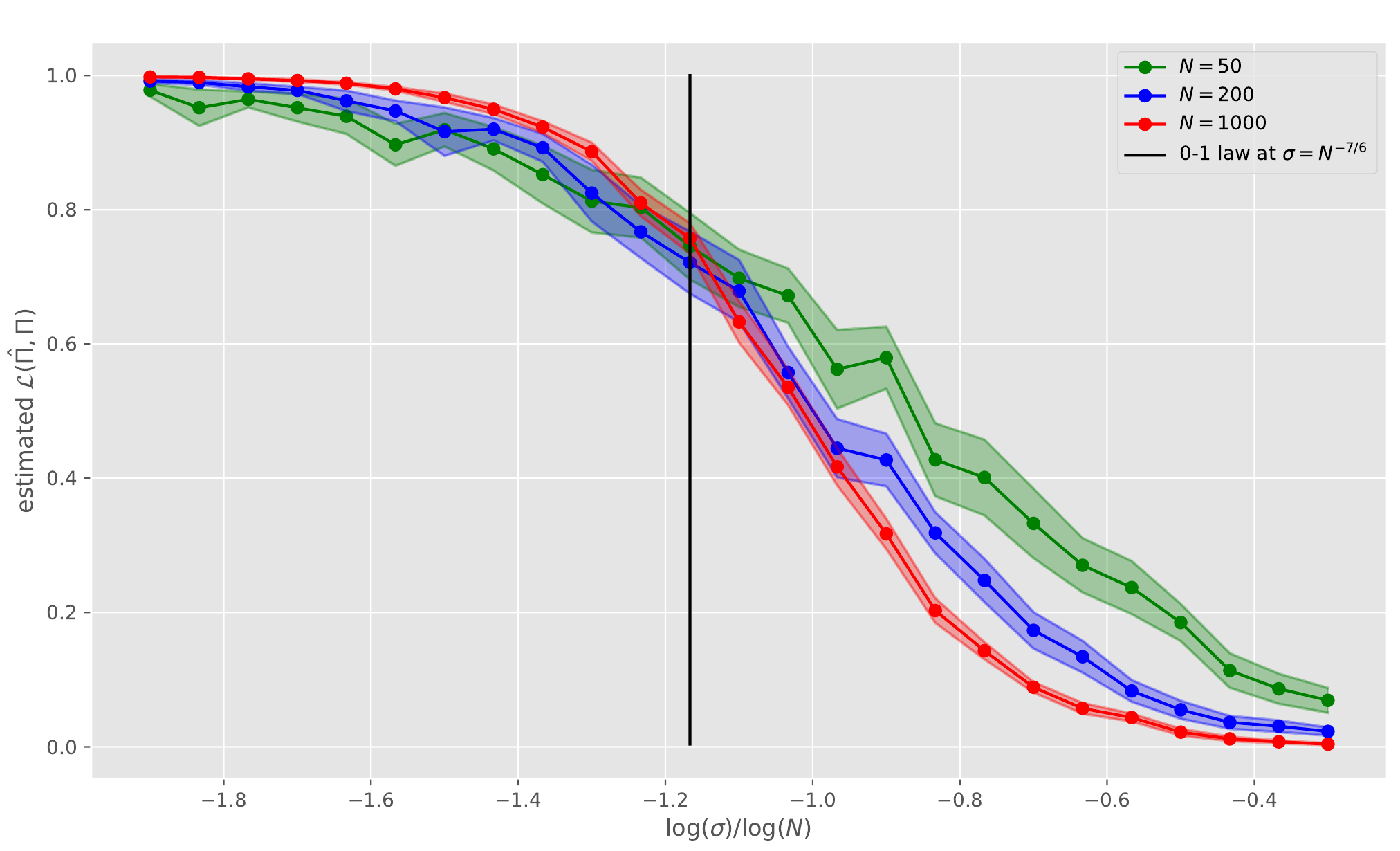}
	\caption{\label{image_overlap} Estimated overlap $\mathcal{L}(\hat{\Pi},\Pi)$ reached by \emph{EIG1} in model \eqref{GOEmodel}, for varying $N$ and $\sigma$. \footnotesize{With $95\%$ confidence intervals}.}
\end{figure}

\begin{remark}
	We can now underline that without loss of generality, we can assume that $\Pi = \mathrm{Id}$, the identity mapping. Indeed, one can return to the general case applying transformations $A \rightarrow \Pi A \Pi^T$ and $H \rightarrow \Pi H \Pi^T$. From now on we will assume in the rest of the paper that $\Pi = \mathrm{Id}$.
\end{remark}

In order to prove this theorem, it is necessary to establish two intermediate results along the way, which could also be of independent interest. First, we study the behavior of $v'_1$ with respect to $v_1$, showing that under some conditions on $\sigma$ and $N$, the difference $v_1 - v'_1$ can be approximated by a renormalized Gaussian standard vector, multiplied by a  variance term $\mathbf{S}$, where $\mathbf{S}$ is a random variable which behavior is well understood in terms of $N$ and $\sigma$ when $N \to \infty$. For this we work under the following assumption: 
\begin{equation}
\label{microscopicregime}
\exists \, \alpha >0, \;\sigma = o \left(N^{-1/2-\alpha}\right),
\end{equation} 

\begin{proposition} \label{prop_gaussian_decomp}
Under assumption \eqref{microscopicregime}, there exists a standard Gaussian vector $Z \sim \mathcal{N}\left(0,I_N\right)$ independent from $v_1$ and a random variable $\mathbf{S} \asymp \sigma N^{1/6}$, such that
	\begin{equation*}
	v'_1 = \left(1 + o_{\mathbb{P}}(1)\right) \left(v_1+ \mathbf{S} \frac{Z}{\|Z\|}\right).
	\end{equation*}
\end{proposition}

\begin{remark} \label{remark_assumption6} This assumption (\ref{microscopicregime}) (or a tighter formulation) arises when studying the diffusion trajectories of eigenvalues and eigenvectors in random matrices, and corresponds to the \textit{microscopic regime} in \cite{Allez14}. This assumption ensures that all eigenvalues of $B$ are close enough to the eigenvalues of $A$. This comparison term is justified from the random matrix theory ($N^{-1/2}$ is the typical amplitude of the spectral gaps $\sqrt{N}(\lambda_i - \lambda_{i+1})$ in the bulk, which are the smaller ones). 

Eigenvectors diffusions in similar models (diffusion processes dawn with the scaling $\sigma = \sqrt{t}$) are studied in \cite{Allez14}, where the main tool is the Dyson Brownian motion (see e.g. \cite{Anderson09}) and its formulation for eigenvectors trajectories, giving stochastic differential equations for the evolutions of $v'_j(t)$ with respect to vectors $v_i=v'_i(0)$. These equations lead to a system of stochastic differential equations for the overlaps $\langle v_i, v'_j(t) \rangle$, which is quite difficult to analyze rigorously. In this work a more elementary method to get a expansion of $v'_1$ around $v_1$, for which this very condition (\ref{microscopicregime}) also appears.

Note that here, spectral gaps at the edge are of order $N^{-1/6}$ so assumption \eqref{microscopicregime} may not optimal for our study, and we expect Proposition \ref{prop_gaussian_decomp} to hold up to $\sigma = o \left(N^{-1/6-\alpha}\right)$. However, since the positive result of Theorem \ref{theorem_01law} holds in a way more restrictive regime -- see condition $(i)$, condition \eqref{microscopicregime} is enough for our purpose and allows a short and simple proof.
\end{remark}

Proposition \ref{prop_gaussian_decomp} suggests the study of $v'_1$ as a Gaussian perturbation of $v_1$. The main question is now formulated as follows: \textit{what is the probability that the perturbation on $v_1$ has an impact on the overlap of the estimator $\hat{\Pi}$ from the \emph{EIG1} method?} To answer this question, we introduce a correlated Gaussian vectors model (or \emph{toy model} hereafter) of parameters $N$ and $s>0$. In this model, we draw a standard Gaussian vector $X$ of size $N$ and $Y= X + s Z$ where $Z$ is an independent copy of $X$. We will use the notation $(X,Y) \sim \mathcal{J}(N,s)$.

Define $r_1$ the function that associates to any vector $T=(t_1,\ldots,t_p)$ the rank of $t_1$ in $T$ (for the usual decreasing order). For $(X,Y) \sim \mathcal{J}(N,s)$ we evaluate
\begin{equation*}
p(N,s) := \mathbb{P}\left(r_1(X)=r_1(Y)\right).
\end{equation*} Our second result shows that there is a zero-one law for the property of rank preservation in the toy model $\mathcal{J}(N,s)$.

\begin{proposition}[Zero-one law for $p(N,s)$]
\label{prop_zero_one_toy}
In the correlated Gaussian vectors model we have the following:
	\begin{itemize}
		\item[$(i)$] If $s = o(1/N)$ then $$p(N,s) \underset{N \to \infty}{\longrightarrow} 1.$$
		\item[$(ii)$] If $s = \omega(1/N)$ then $$p(N,s) \underset{N \to \infty}{\longrightarrow} 0.$$
	\end{itemize}
\end{proposition} These results are illustrated on Figure \ref{image_jouet}, showing the zero-one law at $s \asymp N^{-1}$.

\begin{figure}
	\centering
	\includegraphics[width=14cm,height=8cm]{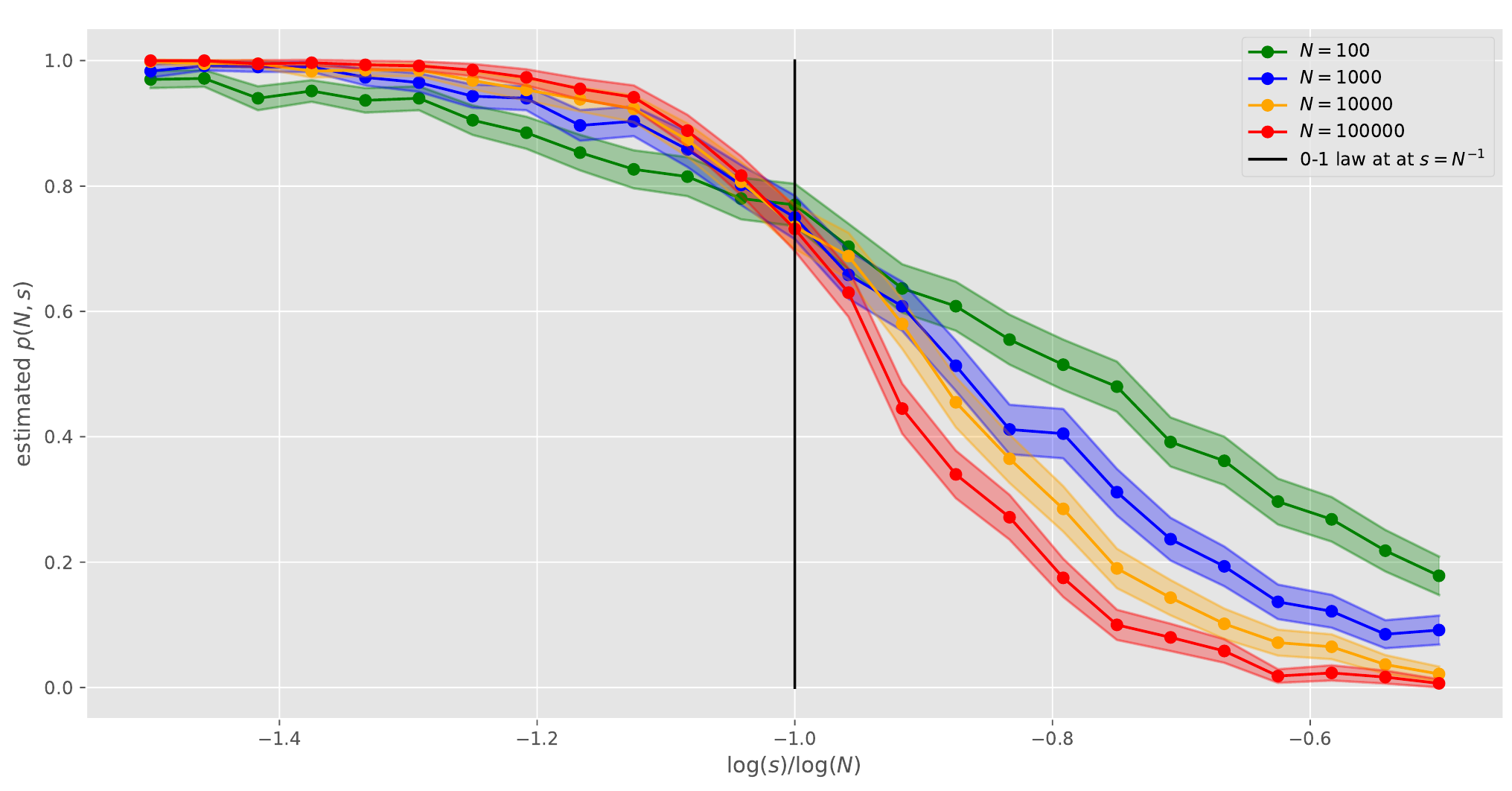}
	\caption{\label{image_jouet} Estimated $p(N,s)$ in the toy model $\mathcal{J}(N,s)$. \footnotesize{With $95\%$ confidence intervals}.}
\end{figure}

\paragraph*{Paper organization} The gaussian approximation of $v_1 - v'_1$ is established in Section \ref{linkGOEtoy} with the proof of Proposition \ref{prop_gaussian_decomp}. The toy model defined here above is studied in Section \ref{toymodel} where Proposition \ref{prop_zero_one_toy} is established. Finally, we gather results of Propositions \ref{prop_gaussian_decomp} and \ref{prop_zero_one_toy} in Section \ref{EIG1threshold} to show Theorem \ref{theorem_01law}. Some additional proofs are deferred to Appendices \ref{section3_add_proofs} and \ref{section45_add_proofs}.

\section{Behavior of the leading eigenvectors of correlated matrices}
\label{linkGOEtoy}
The main idea of this section is to find a first order expansion of $v'_1$ around $v_1$. Recall that we use the notations $\left(v_1, v_2, \ldots, v_N\right)$ for normalized eigenvectors of $A$, corresponding to the eigenvalues $\lambda_1~\geq~\lambda_2~\geq~\ldots~\geq~\lambda_N$. Similarly, $\left(v'_1, v'_2, \ldots, v'_N\right)$ and $\lambda'_1~\geq~\lambda'_2~\geq~\ldots~\geq~\lambda'_N$ will refer to eigenvectors and eigenvalues of $B = A + \sigma H$. Since $A$ and $B$ are symmetric, all these eigenvalues are real and the vectors $\left\{v_i\right\}_i$ (resp. $\left\{v'_i\right\}_i$) are pairwise orthogonal. We also recall that $v'_1$ is taken such that $\langle v_1,v'_1 \rangle >0$.

\subsection{Computation of a leading eigenvector of $B$}
Recall now that we are working under assumption \eqref{microscopicregime}:
\begin{equation*}
\exists \, \alpha >0, \;\sigma = o \left(N^{-1/2-\alpha}\right).
\end{equation*}  Let $w'$ be an (non normalized) eigenvector of $B$ for the eigenvalue $\lambda'_1$ of the form
\begin{equation*}
w' := \sum_{i=1}^{N} \theta_i v_i,
\end{equation*}where we assume that $\theta_1 =1$. Such an assumption can be made a.s. since any hyperplane of $\mathbb{R}^N$ has a null Lebesgue measure in $\mathbb{R}^N$ (see Remark \ref{densitylebesgue}). 

The defining eigenvector equations projected on vectors $v_i$ give
\begin{equation} \label{eq_system_eigenequations}
\left \{
\begin{array}{c c c}
\theta_1 & = & 1,\\
\forall i>1, \; \theta_i &  = &  \dfrac{\sigma}{\lambda'_1 - \lambda_i} \sum_{j=1}^{N} \theta_j \langle H v_j,v_i \rangle, \\
\lambda'_1 - \lambda_1 &  = & \sigma  \sum_{j=1}^{N} \theta_j \langle H v_j,v_1 \rangle. \\
\end{array}
\right.
\end{equation} The strategy is then to approximately solve \eqref{eq_system_eigenequations} with an iterative scheme, leading to the following expansion:

\begin{proposition}
	\label{w'expansion}
	Under the assumption \eqref{microscopicregime} one has the following:
	\begin{equation}
	\label{w'}
	w' = v_1 +  \sigma \sum_{i=2}^{N} \frac{\langle Hv_i,v_1 \rangle }{\lambda_1 - \lambda_i} v_i + o_{\mathbb{P}}\left(\sigma \sum_{i=2}^{N} \frac{\langle Hv_i,v_1 \rangle }{\lambda_1 - \lambda_i} v_i  \right).
	\end{equation}
\end{proposition}
We refer to Appendix \ref{appendix_proof_prop_w'} for the details regarding the definition of the mentioned iterative scheme, as well as a proof of Proposition \ref{w'expansion}. The proof uses assumption \eqref{microscopicregime} an builds upon some standard results on the distribution of eigenvalues in the \emph{GOE}.

\begin{remark}
\label{microscopicregimenotoptimal}
The above proposition could easily be extended for all eigenvectors of $B$, under assumption \eqref{microscopicregime}.
Based on the studies of the trajectories of the eigenvalues and eigenvectors in the GUE \cite{Allez14} and the GOE \cite{Allez14bis}, since we are only interested here in the leading eigenvectors, we expect the result of Proposition \ref{w'expansion} to hold under the weaker assumption $\sigma N^{1/6+\alpha} \to 0$,  for $N^{-1/6}$ is the typical spectral gap $\sqrt{N}(\lambda_1-\lambda_2)$ on the edge. However, as explained before (see Remark \ref{remark_assumption6}), our analysis doesn't require this more optimal assumption. We also know that the expansion \eqref{w'} doesn't hold as soon as $\sigma = \omega(N^{-1/6})$. A result proved by Chatterjee (\cite{Chatterjee14}, Theorem 3.8) shows that the eigenvectors corresponding to the highest eigenvalues $v_1$ of $A$ and $v'_1$ of $B=A+\sigma H$, when $A$ and $H$ are two independent matrices from the GUE, are delocalized (in the sense that $\langle v_1,v'_1 \rangle$ converges in probability to $0$ as $N \to \infty$), when $\sigma = \omega(N^{-1/6})$.
\end{remark}

\subsection{Gaussian representation of $v'_1 - v_1$}
We still work under assumption (\ref{microscopicregime}). After renormalization, we have $v'_1 = \frac{w'}{\| w' \|}$. We are now able to study the behavior of the overlap $\langle v'_1, v_1 \rangle$:
\begin{equation*}
\langle v'_1, v_1 \rangle = \left(1+\sigma^2 (1+o_{\mathbb{P}}(1)) \sum_{i=2}^{N} \dfrac{\langle H v_i,v_1 \rangle ^2 }{\left(\lambda_1 - \lambda_i\right)^2}\right)^{-1/2}
\end{equation*}
Hence
\begin{equation}
\label{eq_diffusion_vp}
\langle v'_1, v_1 \rangle = 1-\frac{\sigma^2}{2} \sum_{i=2}^{N} \frac{\langle H v_i,v_1 \rangle ^2 }{\left(\lambda_1 - \lambda_i\right)^2} + o_{\mathbb{P}}\left(\sigma^2 \sum_{i=2}^{N} \frac{\langle H v_i,v_1 \rangle ^2 }{\left(\lambda_1 - \lambda_i\right)^2} \right).
\end{equation} 

Let us give the heuristic to evaluate the first sum in the right-hand side of \eqref{eq_diffusion_vp}: since the GOE distribution is invariant by rotation (see e.g. \cite{Anderson09}), the random variables $\langle H v_i,v_1 \rangle$ are zero-mean Gaussian, with variance $1/N$. Moreover, it is well known \cite{Anderson09} that the eigenvalue gaps $\lambda_1 - \lambda_i$ are of order $N^{-1/6}$ when $i$ is small, and $N^{-1/2}$ in the bulk (when $i$ is typically of order $N$). These considerations lead to the following:
\begin{lemma}
	\label{lemma_concentrationksi}
	We have the following concentration
	\begin{equation}
	\label{lemma_concentrationksieq}
	\sum_{i=2}^{N} \frac{\langle H v_i,v_1 \rangle ^2 }{\left(\lambda_1 - \lambda_i\right)^2} \asymp N^{1/3}.
	\end{equation}
\end{lemma} We refer to Appendix \ref{proof_lemma_concentrationksi} for a rigorous proof of this result. With this Lemma, we are now able to give the first order expansion of $\langle v'_1, v_1 \rangle$ with respect to $\sigma$:
\begin{equation}
\label{eq_diffusion_vp_2}
\langle v'_1, v_1 \rangle = 1-\frac{\sigma^2}{2} N^{1/3} + o_{\mathbb{P}}\left(\sigma^2 N^{1/3} \right).
\end{equation} 
\begin{remark}
The comparison between $\sigma $ and $N^{1/6}$ made in \cite{Chatterjee14} naturally reappears here, as $\sigma^2 N^{1/3}$ is the typical shift of $v'_1$ with respect to $v_1$.
\end{remark}

The intuition is that the scalar product $\langle v'_1, v_1 \rangle$ is sufficient to derive a Gaussian representation of $v'_1$ w.r.t. $v_1$. We formalize this in the following
\begin{lemma}
	\label{lemma_invrotation}
	Given $v_1$, when writing the decomposition $w' = v_1 + w$, with
	\begin{equation*}
	w :=\sum_{i=2}^{N} \theta_i v_i,
	\end{equation*}
	the distribution of $w$ is invariant by rotation in the orthogonal complement of $v_1$. This implies in particular that given $v_1$, $\|w\|$ and $\frac{w}{\|w\|}$ are independent, and that $\frac{w}{\|w\|}$ is uniformly distributed on $\mathbb{S}^{N-2}$, the unit sphere of $v_1^\perp$.
\end{lemma}
\begin{proof}[Proof of Lemma \ref{lemma_invrotation}]
	We work conditionnally on $v_1$. Let $O$ be an orthogonal transformation of the hyperplane $v_1^{\perp}$ (such that $Ov_1=v_1$). Since the \emph{GOE} distribution is invariant by rotation and $A$ and $H$ are independent, $\widetilde{B} := O^TAO + \sigma O^THO$ has he same distribution as $B = A+\sigma H$.
	
	Note that $Ow' = v_1 + Ow$ is an eigenvector of $\widetilde{B}$ for the eigenvalue $\lambda_{1}$. Since the distribution of the matrix of eigenvectors $(v_2, \ldots, v_n)$ is the Haar measure on the orthogonal group $\mathcal{O}_{n-1}\left(v_1^{\perp}\right)$, denoted by $d\mathcal{H}$, the distribution of $w$ is also invariant by rotation in the orthogonal complement of $v_1$. Furthermore, for any $f,g$ bounded continuous functions and $O \in \mathcal{O}_{n-1}\left(v_1^{\perp}\right)$,
	\begin{flalign*}
	\mathbb{E}\left[f(\|w\|)g\left(\frac{w}{\|w\|}\right)\right] &= \mathbb{E}\left[f(\|w\|)g\left(\frac{Ow}{\|Ow\|}\right)\right] = \mathbb{E}\left[f(\|w\|) \int_{\mathcal{O}_{n-1}\left(v_1^{\perp}\right)} d\mathcal{
		H}(O) g\left(\frac{Ow}{\|Ow\|}\right)\right] \\
	&= \mathbb{E}\left[f(\|w\|) \int_{\mathbb{S}^{n-2}} \frac{g(u) du}{\mathrm{Vol}\left(\mathbb{S}^{n-2}\right)}\right] = \mathbb{E}\left[f(\|w\|)\right]\mathbb{E}\left[g\left(\frac{w}{\|w\|}\right)\right].
	\end{flalign*} 
	This completes the proof of Lemma  \ref{lemma_invrotation}.
\end{proof} 

We can now show the main result of this section, Proposition \ref{prop_gaussian_decomp}.
\begin{proof}[Proof of Proposition \ref{prop_gaussian_decomp}]
	Recall the decomposition $w'= v_1 + w$ with $w =\sum_{i=2}^{N} \theta_i v_i$. According to Lemma \ref{lemma_invrotation}, conditioned to $v_1$, $\frac{w}{\|w\|}$ is uniformly distributed on $\mathbb{S}^{N-2}$, the unit sphere of $v_1^\perp$. We now state a classical result about sampling uniform vectors on a sphere:
	\begin{lemma}
		Let $E$ be $p-$dimensional Euclidean space, endowed with an orthogonal basis $\cB = (e_1, \ldots, e_p)$. Let $u$ be a random vector uniformly distributed on the unit sphere $\mathbb{S}^{p-1}$ of $E$. Then, in basis $\cB$, $u$ has the same distribution as
		$$ \left( \frac{\xi_1}{\sqrt{\sum_{i=1}^{p} \xi_i^2}}, \ldots, \frac{\xi_p}{\sqrt{\sum_{i=1}^{p} \xi_i^2}} \right), $$
		where $\xi_1, \ldots, \xi_p$ are i.i.d. standard normal random variables.
	\end{lemma} We refer e.g. to \cite{ORourke16}, Lemma 10.1, for the proof of this result. In our context, this proves that the joint distribution of the coordinates $w_2, \ldots, w_n$ of $w$ along $v_2, \ldots, v_n$ is always that of a normalized standard Gaussian vector (on $\mathbb{R}^{N-1}$). This joint probability does not dependent on $v_1$. Hence, there exist $Z_2, \ldots, Z_N$ standard Gaussian independent variables, independent from $v_1$ (and from $\| w\|$ by Lemma \ref{lemma_invrotation}), such that:
	\begin{equation*}
	w' = v_1 + \frac{\|w\|}{\left(\sum_{i=2}^{N} Z_i^2\right)^{1/2}} \sum_{i=2}^{N} Z_i v_i.
	\end{equation*}
	Let $Z_1$ be another standard Gaussian variable, independent from everything else. Then
	\begin{equation*}
	w' = \left(1 - \frac{\|w\| Z_1}{\left(\sum_{i=2}^{N} Z_i^2\right)^{1/2}}\right) v_1 + \frac{\|w\|}{\left(\sum_{i=2}^{N} Z_i^2\right)^{1/2}} \sum_{i=1}^{N} Z_i v_i.
	\end{equation*} Let $Z =\sum_{i=1}^{N} Z_i v_i$, which is a standard Gaussian vector. Since the distribution of $Z$ is invariant by permutation of the $\left(Z_i\right)_{1 \leq i \leq N}$, $Z$ and $v_1$ are independent. We have
	\begin{flalign*}
	v'_1 &= \frac{w'}{\| w'\|} = \frac{w'}{\sqrt{1+\|w\|^2}} \\&= \frac{1}{\sqrt{1+\|w\|^2}} \left(1 - \frac{\|w\| Z_1}{\left(\sum_{i=2}^{N} Z_i^2\right)^{1/2}}\right) v_1+ \frac{\|w\| \|Z\|}{\sqrt{1+\|w\|^2} \left(\sum_{i=2}^{N} Z_i^2\right)^{1/2}} \frac{Z}{\|Z\|}.
	\end{flalign*} 
	Taking 
	\begin{equation*}
	\mathbf{S}=\frac{\|w\| \|Z\|}{\left(\sum_{i=2}^{N} Z_i^2\right)^{1/2} - \|w\| Z_1},
	\end{equation*} we get 
	\begin{equation}\label{eq_v'1_wZ}
	v'_1 = \frac{1}{\sqrt{1+\|w\|^2}} \left(1 - \frac{\|w\| Z_1}{\left(\sum_{i=2}^{N} Z_i^2\right)^{1/2}}\right) \left(v_1+ \mathbf{S} \frac{Z}{\|Z\|}\right).
	\end{equation}
	Proposition \ref{w'expansion} together with Lemma \ref{lemma_concentrationksi} yield
	\begin{equation*}
		\|w\|^2=\|w'-v_1\|^2 = (1+o_{\mathbb{P}}(1)) \cdot \sigma^2 \sum_{i=2}^{N} \frac{\langle H v_i,v_1 \rangle ^2 }{\left(\lambda_1 - \lambda_i\right)^2} \asymp \sigma^2 N^{1/3},
	\end{equation*} the last quantity being $o(1)$ under assumption \eqref{microscopicregime}. With the previous computation, equation \eqref{eq_v'1_wZ} becomes
	\begin{align*}
	v'_1 = \left(1 + o_{\mathbb{P}}(1)\right) \left(v_1+ \mathbf{S} \frac{Z}{\|Z\|}\right),
	\end{align*}
	with $\mathbf{S} = (1+o_{\mathbb{P}}(1)) \|w\|\; {\asymp} \; \sigma N^{1/6}$.
\end{proof}

\section{Definition and analysis of a toy model}
\label{toymodel}
Now that we have established a expansion of $v'_1$ with respect to $v_1$, our main question boils down to the study of the effect of a random Gaussian perturbation of a Gaussian vector in terms of rank of its coordinates: if these ranks are preserved, the permutation that aligns these two vectors will be $\hat{\Pi}=\Pi = \mathrm{Id}$. Otherwise we want to understand the error made between $\hat{\Pi}$ and ${\Pi}=\mathrm{Id}$.

\subsection{Definitions and notations}
\label{relevantlink}
We refer to Section \ref{notations_def} for the definition of the toy model $\mathcal{J}(N,s)$. Recall that we want to compute, when $(X,Y) \sim \mathcal{J}(N,s)$, the probability
\begin{equation*}
p(N,s) := \mathbb{P}\left(r_1(X)=r_1(Y)\right).
\end{equation*}
In this section, we denote by $E$ the probability density function of a standard Gaussian variable, and $F$ its cumulative distribution function. Namely
\begin{equation*}
E(u) := \frac{1}{\sqrt{2 \pi}} e^{-u^2 /2} \quad \mbox{and} \quad F(u) := \frac{1}{\sqrt{2 \pi}} \int_{- \infty}^{u} e^{-z^2 /2} dz.
\end{equation*}

We hereafter elaborate on the link between this toy model and our first matrix model (\ref{GOEmodel}) in Section \ref{linkGOEtoy}. Since $v_1$ is uniformly distributed on the unit sphere, we have the equality in distribution $v_1 = \frac{X}{\|X\|}$ where $X$ is a standard Gaussian vector of size $N$, independent of $Z$ by Proposition \ref{prop_gaussian_decomp}. We write
\begin{align*}
v_1 &= \frac{X}{\|X\|}, \\
v'_1 &=  \left( 1 + o_{\mathbb{P}}(1) \right)\left(  \frac{X}{\|X\|} + \mathbf{S} \frac{Z}{\|Z\|}\right).
\end{align*}
Note that for all $\lambda>0$, $r_1(\lambda T) = r_1(T)$, hence
\begin{equation}
\label{GOEtoylink}
r_1(v_1) = r_1(X), \quad r_1(v'_1) = r_1 \left(X + \mathbf{s} Z \right),\\
\end{equation} where 
\begin{equation*}
\mathbf{s}  = \frac{\mathbf{S}  \|X\|}{\|Z\|} \asymp \sigma N^{1/6},
\end{equation*} where we used the law of large numbers ($\|X\|/\|Z\| \to 1$ p.s.) as well as Proposition \ref{prop_gaussian_decomp} in the last expansion. Equation \eqref{GOEtoylink} shows that this toy model is relevant for our initial problem, up to the fact that the noise term $\mathbf{s}$ is random in the matrix model (though we know its order of magnitude to be $\asymp \sigma N^{1/6}$).
\begin{remark}
The intuition for the zero-one law for $p(N,s)$ is as follows. If we sort the $N$ coordinates of $X$ on the real axis, all coordinates being typically perturbed by a factor $s$, it seems natural to compare $s$ with the typical gap between two coordinates of order $1/N$ to decide whether the rank of the first coordinate of $X$ is preserved in $Y$. 
\end{remark}
Let us show that this intuition is rigorously verified. For every couple $(x,y)$ of real numbers, define
\begin{equation*}
\mathcal{N}_{N,s}^{+}(x,y) := \sharp \left\lbrace 1 \leq i \leq N, \; X_i > x, Y_i < y \right\rbrace,
\end{equation*}
\begin{equation*}
\mathcal{N}^{-}_{N,s}(x,y) := \sharp \left\lbrace 1 \leq i \leq N, \; X_i < x, Y_i > y \right\rbrace.
\end{equation*} In the following, we omit all dependencies in $N$ and $s$, using the notations $\mathcal{N}^{+}$ and $\mathcal{N}^{-}$. The corresponding regions are shown on Figure \ref{imageN+-}.
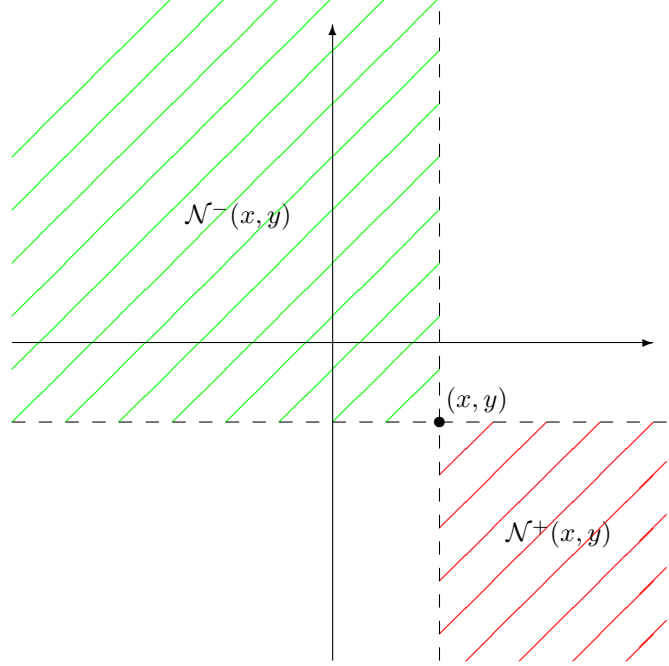
\begin{figure}[H]
	\centering
	\begin{picture}(240,240)
\put(140,90){\textcolor{green}{\line(1,1){20}}}
\put(120,90){\textcolor{green}{\line(1,1){40}}}
\put(100,90){\textcolor{green}{\line(1,1){60}}}
\put(80,90){\textcolor{green}{\line(1,1){80}}}
\put(60,90){\textcolor{green}{\line(1,1){100}}}
\put(40,90){\textcolor{green}{\line(1,1){120}}}
\put(20,90){\textcolor{green}{\line(1,1){140}}}
\put(0,90){\textcolor{green}{\line(1,1){160}}}
\put(0,110){\textcolor{green}{\line(1,1){140}}}
\put(0,130){\textcolor{green}{\line(1,1){120}}}
\put(0,150){\textcolor{green}{\line(1,1){100}}}
\put(0,170){\textcolor{green}{\line(1,1){80}}}
\put(0,190){\textcolor{green}{\line(1,1){60}}}
\put(65,165){\mbox{$\mathcal{N}^{-}(x,y)$}}

\put(160,70){\textcolor{red}{\line(1,1){20}}}
\put(160,50){\textcolor{red}{\line(1,1){40}}}
\put(160,30){\textcolor{red}{\line(1,1){60}}}
\put(160,10){\textcolor{red}{\line(1,1){80}}}
\put(170,0){\textcolor{red}{\line(1,1){75}}}
\put(190,0){\textcolor{red}{\line(1,1){55}}}
\put(210,0){\textcolor{red}{\line(1,1){35}}}
\put(230,0){\textcolor{red}{\line(1,1){15}}}
\put(185,45){\mbox{$\mathcal{N}^{+}(x,y)$}}

\put(0,120){\vector(1,0){240}}

\put(120,0){\vector(0,1){240}}

\put(160,90){\circle*{4}}
\put(174,98){\makebox(0,0){$(x,y)$}}

\multiput(0,90)(10,0){25}{\line(1,0){5}}

\multiput(160,0)(0,10){25}{\line(0,1){5}}

\end{picture}
	\caption{\label{imageN+-}Areas corresponding to $\mathcal{N}^{+} (x,y)$ and $\mathcal{N}^{-} (x,y)$.}
\end{figure} We will also need the following probabilities
\begin{align*}
S^{+}(x,y) &:= \mathbb{P}\left(X_1 > x, Y_1 < y\right), \mbox{ and}\\
S^{-}(x,y) &:= \mathbb{P}\left(X_1 < x, Y_1 > y\right) = S^{+}(-x,-y).
\end{align*}
In terms of distribution, the random vector $$\left(\mathcal{N}^{+}(x,y), \mathcal{N}^{-}(x,y), N-1-\mathcal{N}^{+}(x,y)-\mathcal{N}^{+}(x,y)\right)$$ follows a multinomial distribution of parameters $$\left(N-1,S^{+}(x,y),S^{-}(x,y),1-S^{+}(x,y)-S^{-}(x,y)\right).$$ 
In order to have $r_1(X)=r_1(Y)$, there must be the same number of points on the two domains on Figure \ref{imageN+-}, for $x=X_1$ and $y=Y_1$. We then have the following expression of $p(N,s)$:
\begin{align*} 
p(N,s) &= \mathbb{E}\left[\mathbb{P}\left(\mathcal{N}^{+}(X_1,Y_1)=\mathcal{N}^{-}(X_1,Y_1)\right)\right] \\
&= \int_{\mathbb{R}} \int_{\mathbb{R}} \mathbb{P}(dx,dy) \mathbb{P}(\mathcal{N}^{+}(x,y)=\mathcal{N}^{-}(x,y)) \\
&= \int_{\mathbb{R}} \int_{\mathbb{R}} E(x) E(z) \phi_{x,z}(N,s) \, dx \, dz,
\end{align*} with
\begin{equation}
\label{phiexpression}
\phi_{x,z}(N,s) := \sum_{k=0}^{\lfloor(N-1)/2\rfloor} \binom{N-1}{k}  \binom{N-1-k}{k} \left(S^{+}_{x,z}\right)^k \left(S^{-}_{x,z}\right)^k \left(1-S^{+}_{x,z}-S^{-}_{x,z}\right)^{N-1-2k},
\end{equation} using the notations $S^{+}_{x,z} = S^{+}(x,x+sz)$ and $S^{-}_{x,z} = S^{-}(x,x+sz)$. A simple computation shows that
\begin{flalign}
S^{+}(x,x+sz) &= \int_{x}^{+ \infty} \frac{1}{\sqrt{2\pi}} e^{-u^2 /2} \left( \int_{- \infty}^{z+\frac{x-u}{s}} \frac{1}{ \sqrt{2\pi}} e^{-v^2 / 2} \, dv \right) du \nonumber \\
&= \int_{x}^{+ \infty} E(u) \, F \left(z-\frac{u-x}{s}\right) du,  \label{S+-inf} \\
&= s\int_{0}^{+ \infty} E(x+vs) \, F \left(z-v\right) dv. \label{S+-0} 
\end{flalign}
We have the classical integration result
\begin{equation}
\label{integrateS_+-}
\int_{- \infty}^{z} F(u) du = z F(z) + E(z).
\end{equation}
From (\ref{S+-inf}), (\ref{S+-0}) and (\ref{integrateS_+-}) we derive the following easy lemma:
\begin{lemma}
	\label{s+-}
	For all $x$ and $z$,
	\begin{flalign*}
	S^{+}(x,x+s z) & \underset{s \to 0}{=} s \left[E(x) \left(z F(z) +E(z)\right)\right] + o(s),\\
	S^{+}(x,x+s z) & \underset{s \to \infty}{\longrightarrow} F(z) \left(1-F(x)\right),\\
	S^{-}(x,x+s z) & \underset{s \to 0}{=} s \left[E(x) \left(-z+ z F(z) +E(z)\right)\right] + o(s),\\
	S^{-}(x,x+s z) & \underset{s \to \infty}{\longrightarrow} F(x) \left(1-F(z)\right).
	\end{flalign*}
	Moreover, both $s \mapsto S^{+}(x,x+s z)$ and $s \mapsto S^{-}(x,x+s z)$ are increasing.
\end{lemma}

\subsection{Zero-one law for $p(N,s)$} 
In this Section we give a proof of Proposition \ref{prop_zero_one_toy}.

\begin{proof}[Proof of Proposition \ref{prop_zero_one_toy}]
	\underline{In the first case $(i)$}, if $s = o(1/N)$, we have the following inequality
	\begin{equation}
	\label{pi}
	p(N,s) \geq \int_{\mathbb{R}} \int_{\mathbb{R}} dx dz E(x) E(z) \mathbb{P} \left(\mathcal{N}^{+}(x,x+sz)=\mathcal{N}^{-}(x,x+sz)=0\right).
	\end{equation}
	According to Lemma \ref{s+-}, for all $x,z \in \mathbb{R}$
	\begin{flalign*}
	\mathbb{P} \left(\mathcal{N}^{+}(x,x+sz)=\mathcal{N}^{-}(x,x+sz)=0\right)&= \left(1-S^{+}(x,x+sz)-S^{-}(x,x+sz)\right)^{N-1}\\ & \sim \exp\left(- NsE(x)\left[z(2F(z)-1)+2E(z)\right]\right) \\
	& \underset{N \to \infty}{\longrightarrow} 1,&&
	\end{flalign*} By applying the dominated convergence theorem in (\ref{pi}), we conclude that $p(N,s) \to 1$.\\
	
	\underline{In the second case $(ii)$}, if $s N \to \infty$, recall that
	\begin{equation}
	\label{pii}
	p(N,s)= \int_{\mathbb{R}} \int_{\mathbb{R}} dx dz E(x) E(z) \phi_{x,z}(N,s),
	\end{equation} with $\phi_{x,z}$ defined in equation $(\ref{phiexpression})$. In the rest of the proof, we fix $x$ and $z$ two real numbers. Letting
\begin{equation*}
b(N,s,k) := \binom{N-1}{k} \left(S^{+}_{x,z}\right)^k \left(1- S^{+}_{x,z}\right)^{N-1-k}
\end{equation*} and 
\begin{equation*}
M(N,s) := \underset{0 \leq k \leq N-1}{\max} b(N,s,k).
\end{equation*} Note that by Lemma \ref{s+-}, there exists $C = C(x,z) <1$ such that for $N$ large enough, $S^{+}_{x,z}<C<1$. Moreover, combining this Lemma with assumption $(ii)$ gives that $ N S^{+}_{x,z} \to \infty$. It is also known that $M(N,s) = b(N,s,\lfloor N S^{+}_{x,z} \rfloor)$ and a classical computation shows that in this case (see e.g. \cite{Bollobas2001}, formula 1.5):
\begin{flalign*}
M(N,s) & = \binom{N-1}{\lfloor N S^{+}_{x,z} \rfloor} \left(S^{+}_{x,z}\right)^{\lfloor N S^{+}_{x,z} \rfloor} \left(1- S^{+}_{x,z}\right)^{N-1-{\lfloor N S^{+}_{x,z} \rfloor}}\\
& \sim \frac{1}{\sqrt{2 \pi N t(1-t)}} t^{-(N-1)t} (1-t)^{-(N-1)(1-t)} \left(S^{+}_{x,z}\right)^{(N-1)t} \left(1- S^{+}_{x,z}\right)^{(N-1)(1-t)}\\
& = \left(N S^{+}_{x,z}\right)^{-1/2} (1+O(1)) \to 0.
\end{flalign*} where $t := \frac{\lfloor N S^{+}_{x,z} \rfloor}{N-1} \sim  S^{+}_{x,z}$. Working with equation \eqref{phiexpression}, we obtain the following control
\begin{flalign*}
\phi_{x,z}(N,s) & \leq M(N,s) \times \sum_{k=0}^{\lfloor(N-1)/2\rfloor}  \binom{N-1-k}{k} \left(S^{-}_{x,z}\right)^k \frac{\left(1-S^{+}_{x,z}-S^{-}_{x,z}\right)^{N-1-2k}}{\left(1-S^{+}_{x,z}\right)^{N-1-k}}\\
& \overset{(a)}{=} M(N,s) \times \frac{(1-S^{+}_{x,z})\left(1-\left(\frac{-S^{+}_{x,z}}{1-S^{-}_{x,z}} \right)^N\right)}{1+S^{-}_{x,z}-S^{+}_{x,z}}\\
& \overset{(b)}{=} M(N,s) \times O(1) \underset{N \to \infty}{\longrightarrow} 0. &&
\end{flalign*}
We used in $(b)$ the fact that $S^{+}_{x,z}+S^{-}_{x,z}$ is increasing in $s$, and that given $x$ and $z$, for all $s>0$, by Lemma \ref{s+-}, $$S^{+}_{x,z}+S^{-}_{x,z}< F(x)\left(1-F(z)\right)+F(z)\left(1-F(x)\right)<1.$$ We used in $(a)$ the following combinatorial result:

\begin{lemma}
\label{fibo}
For all $\alpha>0$, 
\begin{equation}
\label{sum_fibo}
\sum_{k=0}^{\lfloor(N-1)/2\rfloor}  \binom{N-1-k}{k} \alpha^k = \frac{1}{\sqrt{1+4 \alpha}} \left[\left(\frac{1+\sqrt{1+4 \alpha}}{2}\right)^N - \left(\frac{1-\sqrt{1+4 \alpha}}{2}\right)^N\right].
\end{equation}
\end{lemma} 
We refer to Appendix \ref{proof_lemma_fibo} for a proof of this result. To obtain $(a)$ from Lemma \ref{fibo}, we apply $(\ref{sum_fibo})$ to $\alpha = \frac{S^{-}_{x,z} \left(1-S^{+}_{x,z}\right)}{\left(1-S^{+}_{x,z}-S^{-}_{x,z}\right)^2}$, with $\sqrt{1+4 \alpha} = \frac{1-S^{+}_{x,z}+S^{-}_{x,z}}{1-S^{+}_{x,z}-S^{-}_{x,z}}$. Some simple simplifications then give the claimed result. The dominated convergence theorem in $(\ref{pii})$ shows that $p(N,s) \to 0$ and ends the proof.
\end{proof}

\begin{remark}
The above computations also imply the existence of a non-degenerate limit of $p(N,s)$ in the critical case where $s N \to c>0$: in this case, previous discussions as well as Lemma \ref{s+-} show that the joint distribution of $(\cN^+(x,x+sz), \cN^-(x,x+sz))$ is asymptotically
$$\Poi(c \left[E(x) \left(z F(z) +E(z)\right)\right]) \otimes \Poi(c \left[E(x) \left(-z+z F(z) +E(z)\right)\right]).  $$ 
Therefore, $p(N,s)$ has a non-degenerate limit given by 
\begin{equation}
\int_{\mathbb{R}} \int_{\mathbb{R}} E(x) E(z) \cdot \mathbf{G}\left(c \left[E(x) \left(z F(z) +E(z)\right)\right],c \left[E(x) \left(-z+z F(z) +E(z)\right)\right] \right) \, dx \, dz,
\end{equation} where
\begin{equation}
\mathbf{G}(a,b) := \mathbb{P}(\Poi(a)=\Poi(b)) = e^{-(a+b)} \sum_{k \geq 0} \frac{a^k b^k}{(k!)^2}.
\end{equation} 
\end{remark}

\section{Analysis of the \emph{EIG1} method for matrix alignment}
\label{EIG1threshold}
By now, we come back to our initial problem, which is the analysis of \emph{EIG1} method. Recall that for any estimator $\hat{\Pi}$ of $\Pi$, its overlap is defined as follows
\begin{equation*}
\mathcal{L}(\hat{\Pi},\Pi) := \frac{1}{N} \sum_{i=1}^{N} \mathbf{1}_{\hat{\Pi}(i)=\Pi(i)}.
\end{equation*}
The aim of this section is to show how Propositions \ref{prop_gaussian_decomp} and \ref{prop_zero_one_toy} can be assembled to show the main result of our study, namely Theorem \ref{theorem_01law}.

\begin{proof}[Proof of Theorem \ref{theorem_01law}]
\underline{In the first case $(i)$}, assuming $\sigma = o(N^{-7/6-\epsilon})$ for some $\epsilon>0$, then in particular condition (\ref{microscopicregime}) holds. Proposition \ref{prop_gaussian_decomp} as well as equation \eqref{GOEtoylink} in Section \ref{toymodel} enable to identify $v_1$ and $v'_1$ with the following vectors:
\begin{equation}
\label{identificationrang}
v_1 \sim X, \quad v'_1 \sim X + \mathbf{s} Z,
\end{equation} where $X$ and $Z$ are two independent Gaussian vectors from the toy model, and where $\mathbf{s} \asymp \sigma N^{1/6} $  w.h.p. Recall that we work under the assumptions $\Pi=\mathrm{Id}$ and $\langle v_1, v'_1 \rangle >0$. In this case, we expect $\Pi_{+}$ to be very close to $\mathrm{Id}$.

We will use the notations of Section \ref{toymodel} hereafter. Let's take $f \in \mathcal{F}$ such that w.h.p., $\sigma N^{1/6}f(N)^{-1} \leq \mathbf{s} \leq \sigma N^{1/6 }f(N)$. We have for all $1 \leq i \leq N$,
\begin{flalign*}
\mathbb{P}\left({\Pi_{+}}(i)=\Pi(i)\right) &=  \mathbb{P}\left({\Pi_{+}}(1)=\Pi(1)\right) \\
& = \mathbb{E}\left[\iint dx dz E(x) E(z) \phi_{x,z}\left(N, \mathbf{s} \right)\mathbf{1}_{\sigma N^{1/6}f(N)^{-1} \leq \mathbf{s} \leq \sigma N^{1/6 }f(N)}\right] + o(1)\\
& = \iint dx dz E(x) E(z) \mathbb{E}\left[\phi_{x,z}\left(N, \mathbf{s} \right)\mathbf{1}_{\sigma N^{1/6}f(N)^{-1} \leq \mathbf{s} \leq \sigma N^{1/6 }f(N)}\right] + o(1).&&
\end{flalign*} When conditioning on the event $\cA$ where $\sigma N^{1/6}f(N)^{-1} \leq \mathbf{s} \leq \sigma N^{1/6 }f(N)$, we know that $\mathbf{s}N \to 0$ by condition $(i)$ and for all $x,z$, $\dE\left[\phi_{x,z}\left(N, \mathbf{s} \right) \, | \, \cA \right] \to 1$ as shown in Section \ref{toymodel}. Since $\cA$ occurs w.h.p. we have
\begin{flalign*}
\dE\left[\phi_{x,z}\left(N, \mathbf{s} \right)\mathbf{1}_{\cA}\right] {\longrightarrow} 1,
\end{flalign*} which implies with the dominated convergence theorem that 
\begin{equation}\label{eq_last_eq(i)}
\mathbb{E}\left[\mathcal{L}({\Pi_{+}},\Pi)\right] \underset{N \to \infty}{\longrightarrow} 1
\end{equation}
and thus
\begin{equation*}
\mathcal{L}({\Pi_{+}},\Pi) \overset{L^1}{\rightarrow} 1.
\end{equation*} 
We now check that w.h.p., $\Pi_{+}$ is preferred to $\Pi_{-}$ in the \emph{EIG1} method:
\begin{lemma}
		\label{Pi+casei}
		In the case $(i)$, if $\langle v_1,v'_1 \rangle >0$, we have w.h.p.
		\begin{equation*}
		\langle A, \Pi_{+} B \Pi_{+}^T \rangle > \langle A, \Pi_{-} B \Pi_{-}^T \rangle,
		\end{equation*} in other words Algorithm \emph{EIG1} returns w.h.p. $\hat{\Pi}=\Pi_{+}$.
\end{lemma} This Lemma is proved in Appendix \ref{proof_lemma_Pi+(i)} and implies, together with \eqref{eq_last_eq(i)}, that
\begin{flalign*}
\mathbb{E}\left[\mathcal{L}(\hat{\Pi},\Pi)\right] & \geq \mathbb{E}\left[\mathcal{L}(\hat{\Pi},\Pi)\mathbf{1}_{\hat{\Pi}=\Pi_{+}}\right] = \mathbb{E}\left[\mathcal{L}(\Pi_{+},\Pi)\mathbf{1}_{\hat{\Pi}=\Pi_{+}}\right] \\
& = \mathbb{E}\left[\mathcal{L}(\Pi_{+},\Pi)\right]  - \mathbb{E}\left[\mathcal{L}(\Pi_{+},\Pi)\mathbf{1}_{\hat{\Pi}=\Pi_{-}}\right] \\
& = 1 -o(1).
\end{flalign*} and thus 
\begin{equation}\label{eq_conv_i}
\mathcal{L}(\hat{\Pi},\Pi) \underset{N \to \infty}{\overset{L^1}{\longrightarrow}} 1.
\end{equation} 

\underline{In the second case $(ii)$}, if condition (\ref{microscopicregime}) is verified then the identification (\ref{identificationrang}) still holds and the proof of case $(i)$ adapts well. However, if (\ref{microscopicregime}) is not verified, we can still make a link with the toy model studied in Section \ref{toymodel}. Let's use a simple coupling argument: if $\sigma = \omega(N^{-1/2-\alpha})$ for some $\alpha \geq 0$, let's take $\sigma_1, \sigma_2 >0$ such that $$ \sigma^2 = \sigma_1^2 + \sigma_2^2 $$ and $$ N^{-7/6+\epsilon} \ll \sigma_1 \ll N^{-1/2-\alpha}, $$ fixing for instance $\sigma_1 = N^{-1}$. We will use the notation $\widetilde{v}_1$, now viewed as the leading eigenvector of the matrix 
\begin{equation*}
\widetilde{B} = A + \sigma_1 H + \sigma_2 \widetilde{H},
\end{equation*} where $\widetilde{H}$ is an independent copy of $H$. This has no consequence in terms of distribution : $(A,\widetilde{B})$ is still drawn under model $(\ref{GOEmodel})$. Let's denote $v'_1$ the leading eigenvector of $B_1=A+\sigma_1 H$, chosen so that $\langle v_1,v'_1\rangle >0$. It is clear that condition \eqref{microscopicregime} holds for $\sigma_1$. We have the following result, based on the invariance by rotation of the \emph{GOE} distribution:

\begin{lemma}
	\label{stillthelink}
	We still have the following equality in distribution:
	\begin{equation*}
	\left(r_1(v_1),r_1(\widetilde{v}_1)\right) \overset{(d)}{=} \left(r_1(X),r_1(X+\mathbf{s}Z) \right),
	\end{equation*} 
	where $X$, $Z$ are two standard Gaussian vectors from the toy model, with w.h.p. 
	\begin{equation*}
	\mathbf{s} \geq \mathbf{s^1} \asymp \sigma_1 N^{1/6}.
	\end{equation*} 
\end{lemma}
We refer to Appendix \ref{proof_stillthelink} for a proof. Since w.h.p. $\mathbf{s} \geq \mathbf{s^1}$ and $\mathbf{s^1} N \asymp \sigma_1 N^{7/6} \to \infty$, we have for all $1 \leq i \leq N$,
\begin{flalign*}
\mathbb{P}\left(\Pi_{+}(i)=\Pi(i)\right) &=  \mathbb{P}\left(\Pi_{+}(1)=\Pi(1)\right) \\
& = \mathbb{E} \left[\iint dx dz E(x) E(z) \phi_{x,z}(N,\mathbf{s}) \mathbf{1}_{\mathbf{s} N  \to \infty} \right] + o(1)\\
& = \iint dx dz E(x) E(z) \mathbb{E} \left[\phi_{x,z}(N,\mathbf{s}) \mathbf{1}_{\mathbf{s} N  \to \infty} \right] + o(1).
\end{flalign*}
With the same arguments as in the case $(i)$, we show that 
$\phi_{x,z}(N,\mathbf{s}) \mathbf{1}_{\mathbf{s} N  \to \infty} \overset{L^1}{\longrightarrow} 0,$ which implies
\begin{equation*}
\mathbb{E}\left[\mathcal{L}(\Pi_{+},\Pi)\right] \underset{N \to \infty}{\longrightarrow} 0,
\end{equation*}
hence $\mathcal{L}(\Pi_{+},\Pi) \underset{N \to \infty}{\overset{L^1}{\longrightarrow}} 0.$ The last step is to verify that the overlap achieved by $\Pi_{-}$ does not outperform that of $\Pi_{+}$. We prove the following Lemma in Appendix \ref{proof_lemma_Pi+caseii}:
\begin{lemma}
	\label{Pi+caseii}
	In the case $(ii)$, if $\langle v_1,v'_1 \rangle >0$, we also have
	\begin{equation*}
	\mathcal{L}(\Pi_{-},\Pi) \underset{N \to \infty}{\overset{L^1}{\longrightarrow}} 0.
	\end{equation*} 
\end{lemma} Lemma \ref{Pi+caseii} then gives
\begin{equation*}
\mathbb{E}\left[\mathcal{L}(\hat{\Pi},\Pi)\right] \leq \mathbb{E}\left[\mathcal{L}(\Pi_{+},\Pi)\right] + \mathbb{E}\left[\mathcal{L}(\Pi_{-},\Pi)\right] \underset{N \to \infty}{\longrightarrow} 0,
\end{equation*}and thus 
\begin{equation}\label{eq_conv_ii}
\mathcal{L}(\hat{\Pi},\Pi) \underset{N \to \infty}{\overset{L^1}{\longrightarrow}} 0.
\end{equation} 

Of course, the convergences in \eqref{eq_conv_i} and \eqref{eq_conv_ii} also hold in probability, by Markov's inequality.
\end{proof}

\appendix

\section{Additionnal proofs for Section \ref{linkGOEtoy}} \label{section3_add_proofs}
Throughout the proofs, all variables denoted by $C_i$ with $i=1,2,\ldots$ are unspecified, independent, positive constants.

\subsection{Proof of Proposition \ref{w'expansion}}
\label{appendix_proof_prop_w'}
\begin{proof}[Proof of Proposition \ref{w'expansion}]
	Let us establish a first inequality: since the GOE distribution is invariant by rotation (see e.g. \cite{Anderson09}), the random variables $\langle H v_j,v_i \rangle$ are zero-mean Gaussian, with variance $1/N$ of $i \neq j$ and $2/N$ if $i=j$. Hence, w.h.p.
	\begin{equation}
	\label{ineq_encadrementgaussiennes}
	\underset{1 \leq i,j \leq N}{\sup} \left|\langle H v_j,v_i \rangle \right|\leq C_1 \sqrt{\frac{\log N}{N}}.
	\end{equation} We will use the following short-hand notation for $1 \leq i,j \leq N$:
	\begin{equation*}
	m_{i,j} := \langle H v_j,v_i \rangle,
	\end{equation*} 
	The defining eigenvector equations projected on vectors $v_i$ write
	\begin{equation}
	\left \{
	\begin{array}{c @{=} c}
	\theta_i & \dfrac{\sigma}{\lambda'_1 - \lambda_i} \sum_{j=1}^{N} \theta_j m_{i,j}, \\
	\lambda'_1 - \lambda_1 & \sigma  \sum_{j=1}^{N} \theta_j m_{1,j}. \\
	\end{array}
	\right.
	\end{equation}
	In order to approximate the $\theta_i$ variables, we define the following iterative scheme:
	\begin{equation}
	\label{schemadepicard}
	\left \{
	\begin{array}{c @{\; = } c}
	\theta_i^k & \dfrac{\sigma}{\lambda_1^{k-1} - \lambda_i} \sum_{j=1}^{N} \theta_j^{k-1} m_{i,j} ,\\
	\lambda^{k}_1 - \lambda_1 & \sigma  \sum_{j=1}^{N} \theta_j^{k-1} m_{1,j} ,
	\end{array}
	\right.
	\end{equation} with initial conditions $\left(\theta_i^{0}\right)_{2 \leq i \leq N}=0$ and $\lambda_{1}^0 = \lambda_1$, and setting $\theta_1^{k}=1$ for all $k$. For $k \geq 1$, define
	\begin{equation*}
	\Delta_k := \sum_{i \geq 2} \left|\theta_i^k - \theta_i^{k-1}\right|,
	\end{equation*} 
	and for $k \geq 0$,
	\begin{equation*}
	S_k := \sum_{i \geq 1} \left|\theta_i^k \right|.
	\end{equation*} Recall that under assumption (\ref{microscopicregime}), there exists $\alpha>0$ such that $\sigma = o \left(N^{-1/2-\alpha}\right)$. We define $\epsilon$ as follows: 
	\begin{equation*}
	\epsilon = \epsilon(N) = \sqrt{\sigma N^{1/2+\alpha}}.
	\end{equation*} The idea is to show that the sequence $\left\{\Delta_k \right\}_{k \geq 1}$ decreases geometrically with $k$ at rate $\epsilon$. More specifically, we show the following result:
	\begin{lemma}
		\label{propagation_picard}
		With the same notations and under the assumption (\ref{microscopicregime}) of Proposition \ref{w'expansion}, one has w.h.p.
	\begin{itemize}
		\item[$(i)$] $\forall k \geq 1, \; \Delta_k \leq \Delta_1 \epsilon^{k-1}$, \\
		\item[$(ii)$] $\forall k \geq 0, \forall \, 2 \leq i \leq N, \; \left|\lambda_{1}^k - \lambda_i\right| \geq \frac{1}{2} \left|\lambda_{1} - \lambda_i\right| \left(1- \epsilon - \ldots - \epsilon^{k-1}\right)$, \\
		\item[$(iii)$] $\forall k \geq 0, \; S_k \leq 1 +(1+\ldots+\epsilon^{k-1}) \Delta_1$,\\
		\item[$(iv)$]  $\sum_{i=2}^{N} \left| \theta_i - \theta_i^{1} \right|^2 = o\left(\sum_{i=2}^{N} \left| \theta_i^{1} \right|^2 \right)$.
	\end{itemize}
	\end{lemma}
	This Lemma is proved in the next section. Equation $(iv)$ of Lemma \ref{propagation_picard} yields
	\begin{flalign*}
	w' & = v_1 + \sum_{i=2}^{N} \theta_i^1 v_i + \sum_{i=2}^{N} \left(\theta_i - \theta_i^1\right) v_i\\
	& = v_1 + \sigma \sum_{i=2}^{N} \frac{\langle Hv_i,v_1 \rangle }{\lambda_1 - \lambda_i} v_i + o_{\mathbb{P}}\left(\sigma \sum_{i=2}^{N} \frac{\langle Hv_i,v_1 \rangle }{\lambda_1 - \lambda_i} v_i  \right). 
	\end{flalign*}
\end{proof}

\subsection*{Proof of Lemma \ref{propagation_picard}}
\begin{proof}[Proof of Lemma \ref{propagation_picard}]
In this proof we will use the same notations as defined in the proof of Proposition \ref{w'expansion}, and we make the assumption $(\ref{microscopicregime})$. We now state three technical lemmas controlling some statistics of eigenvalues in the \emph{GOE} which are useful hereafter.

\begin{lemma}
	\label{lemma_sommevp1}
	W.h.p., for all $\delta>0$,
	\begin{equation}
	\label{sommevp1eq}
	\sum_{j = 2}^{N} \frac{1}{\lambda_1 - \lambda_j} \leq O\left( N ^{1+\delta}\right).
	\end{equation}
\end{lemma}

\begin{lemma}
	\label{lemma_sommevp2}
	We have
	\begin{equation}
	\label{sommevp2eq}
	\; \sum_{j = 2}^{N} \frac{1}{\left(\lambda_1 - \lambda_j \right)^2} \asymp N^{4/3}.
	\end{equation} 
\end{lemma}

\begin{lemma}
	\label{lemma_controletrousp}
	For any $C>0$, w.h.p.
	\begin{equation}
	\label{controletrouspeq}
	\lambda_1 - \lambda_2 \geq N^{-2/3} \left(\log N\right)^{-C \log \log N}.
	\end{equation}
\end{lemma}

Proofs of these three Lemmas can be found in the next sections. We will work under the event (that occurs w.h.p.) on which the equations (\ref{sommevp1eq}), (\ref{sommevp2eq}), (\ref{controletrouspeq}), (\ref{lemma_concentrationksieq}) and (\ref{ineq_encadrementgaussiennes}) are satisfied. We show the following inequalities:
	
	\begin{itemize}
		\item[$(i)$] $\forall k \geq 1, \; \Delta_k \leq \Delta_1 \epsilon^{k-1}$, \\
		\item[$(ii)$] $\forall k \geq 0, \forall \, 2 \leq i \leq N, \; \left|\lambda_{1}^k - \lambda_i\right| \geq \frac{1}{2} \left|\lambda_{1} - \lambda_i\right| \left(1- \epsilon - \ldots - \epsilon^{k-1}\right)$, \\
		\item[$(iii)$] $\forall k \geq 0, \; S_k \leq 1 + (1+\ldots+\epsilon^{k-1}) \Delta_1$.\\
	\end{itemize}
	
	Recall that $\epsilon$ is given by
	\begin{equation*}
	\epsilon = \epsilon(N) = \sqrt{\sigma N^{1/2+\alpha}}.
	\end{equation*}
	
	We will denote by $f_{i}(N)$, with $i$ an integer, functions as defined in Lemma \ref{lemma_sommevp2}. All the following inequality will be valid for $N$ large enough (uniformly in $i$ and in $k$).\\
	
	\textbf{Step 1: propagation of the first equation.}
	Let $k \geq 3$. We work by induction, assuming that $(i)$, $(ii)$ and $(iii)$ are verified until $k-1$. 
	\begin{flalign*}
	\left| \theta_i^k - \theta_i^{ k-1} \right| & \leq \left| \frac{\sigma}{\lambda_1^{k-1}-\lambda_i} \sum_{j=2}^{N} \left(\theta_j^{k-1} - \theta_j^{k-2}\right)m_{i,j} \right|  + \left| \frac{\sigma \left(\lambda_1^{k-2}-\lambda_1^{k-1}\right)}{\left(\lambda_1^{k-1}-\lambda_i\right)\left(\lambda_1^{k-2}-\lambda_i\right)} \sum_{j=1}^{N} \theta_j^{k-2} m_{i,j} \right|\\
	& \leq \frac{ \sigma}{\left|\lambda_1^{k-1}-\lambda_i\right|} C_1 \sqrt{\frac{\log N}{N}} \Delta_{k-1} + \sigma C_1 \sqrt{\frac{\log N}{N}} S_{k-2} \frac{\left|\lambda_1^{k-2}-\lambda_1^{k-1}\right|}{\left|\lambda_1^{k-1}-\lambda_i\right|\left|\lambda_1^{k-2}-\lambda_i\right|} \\ 
	& \overset{(a)}{\leq}  \sigma \frac{3}{\left|\lambda_1-\lambda_i\right|} C_1 \sqrt{\frac{\log N}{N}} \Delta_{k-1} + \sigma C_1 \sqrt{\frac{\log N}{N}} S_{k-2} \frac{9 \left|\lambda_1^{k-2}-\lambda_1^{k-1}\right|}{\left|\lambda_1-\lambda_i\right|^2} \\ 
	& \overset{(b)}{\leq}  \sigma \frac{3}{\left|\lambda_1-\lambda_i\right|} C_1 \sqrt{\frac{\log N}{N}} \Delta_{k-1} + \sigma C_1 \sqrt{\frac{\log N}{N}} 2 \frac{9 \left|\lambda_1^{k-2}-\lambda_1^{k-1}\right|}{\left|\lambda_1-\lambda_i\right|^2} .
	\end{flalign*} We applied $(ii)$ to $k-1, k-2$ in $(a)$ and $(iii)$ to $k-2$ in $(b)$.
	Note that
	\begin{align*}
	\left|\lambda_1^{k-2}-\lambda_1^{k-1}\right| & = \left|\sigma \sum_{j=1}^{N} \left(\theta_j^{k-2}-\theta_j^{k-3}\right) m_{i,j}\right| \leq \sigma C_1 \sqrt{\frac{\log N}{N}} \Delta_{k-2},
	\end{align*} which yields the inequality:
	\begin{align*}
	\left| \theta_i^k - \theta_i^{ k-1} \right| & \leq  \frac{\sigma}{\left|\lambda_1-\lambda_i\right|} f_1(N) N^{-1/2} \Delta_{k-1} + \frac{\sigma^2 }{\left|\lambda_1-\lambda_i\right|^2} f_2(N) N^{-1} \Delta_{k-2}.
	\end{align*}
 	We choose $\delta $ such that $0 < \delta < \alpha$ (where $\alpha$ is fixed by $(\ref{microscopicregime})$), and we sum from $i=2$ to $N$:
	\begin{align*}
	\Delta_k & \leq \sigma f_1(N) N^{1/2 + \delta} \Delta_{k-1} + \sigma^2 f_3(N) N^{1/3} \Delta_{k-2} \\
	& \overset{(a)}{\leq} o(\epsilon) \epsilon^{k-2} \Delta_1 + o(\epsilon^2) \epsilon^{k-3} \Delta_1 \\
	& \leq  \epsilon^{k-1} \Delta_1.
	\end{align*} We used $\sigma f_1(N) N^{1/2 + \delta} =o(\epsilon)$, $\sigma^2 f_3(N) N^{1/3} = o(\epsilon^2)$ and we applied $(i)$ to $k-1$ and $k-2$ in $(a)$.\\
	
	\textbf{Step 2: propagation of the second equation.}	
	Let $k \geq 2$, and $0 < \delta < \alpha$. We work by induction, assuming that $(i)$, $(ii)$ and $(iii)$ are verified until $k-1$. 
	\begin{align*}
	\left| \lambda_1^k - \lambda_1^{ k-1} \right| & \leq \sigma f_1(N) N^{-1/2} \Delta_{k-1} \\ 
	& \overset{(a)}{\leq} \sigma f_1(N) N^{-1/2} {\epsilon^{k-2} \Delta_1}\\
	& \leq N^{-2/3} (\log N)^{-C \log \log N} \epsilon^{k-2} \Delta_1\\
	& \leq \frac{\lambda_1-\lambda_2}{2} \epsilon^{k-2} \Delta_1\\
	& \leq \frac{\lambda_1-\lambda_i}{2} \epsilon^{k-2} \Delta_1.
	\end{align*} We applied $(i)$ to $k-1$ in $(a)$. Note that
	\begin{align*}
	\Delta_1 & = \sum_{j=2}^{N} \frac{\sigma}{\lambda_1 - \lambda_i} \left| m_{i,1} \right| \leq \sigma f_1(N) N^{1/2+\delta} \leq o(\epsilon).
	\end{align*}
	Applying $(ii)$ to $k-1$, we get
	\begin{align*}
	\left| \lambda_1^k - \lambda_i \right| & \geq \left| \lambda_1 - \lambda_1^{k-1} \right| - \left| \lambda_1^k - \lambda_1^{k-1} \right| \\ 
	& \geq \frac{\lambda_1-\lambda_i}{2} \left(1-\epsilon-\ldots-\epsilon^{k-2}\right) - \frac{\lambda_1-\lambda_i}{2} \epsilon^{k-1}\\
	& \geq \frac{\lambda_1-\lambda_i}{2} \left(1-\epsilon-\ldots-\epsilon^{k-1}\right).
	\end{align*}
	
	\textbf{Step 3: propagation of the third equation.} Let $k \geq 1$. Here again, we work by induction, assuming that $(i)$, $(ii)$ and $(iii)$ are verified until $k-1$. 
	\begin{align*}
	S_k &= 1 + \sum_{j=2}^{N} \left|\theta^k_j\right|\\
	& \leq 1 + \Delta_k + S_{k-1} - 1\\
	& \overset{(a)}{\leq} \epsilon^{k-1} \Delta_1 + 1 + \left(1+\ldots+\epsilon^{k-2}\right)\Delta_1\\
	& \leq 1 + \left(1+\epsilon+\ldots+\epsilon^{k-1}\right) \Delta_1.
	\end{align*} We applied $(i)$ to $k$ and $(iii)$ to $k-1$ in $(a)$.
	
	\textbf{Step 4: Proof of $(i)$ for $k=1,2$, $(ii)$ for $k=0,1$ and $(iii)$ for $k=0,1$.}
	The equation $(i)$ for $k=1$ is obvious. For $k=2$ :
	\begin{align*}
	\left| \theta_i^2 - \theta_i^1 \right| & \leq \left| \frac{\sigma}{\lambda_1^1-\lambda_i} \sum_{j=2}^{N} \left(\theta_j^{1} - \theta_j^{0}\right)m_{i,j} \right|  + \left| \frac{\sigma \left(\lambda_1^{0}-\lambda_1^{1}\right)}{\left(\lambda_1^{1}-\lambda_i\right)\left(\lambda_1^{0}-\lambda_i\right)} \sum_{j=1}^{N} \theta_j^{0} m_{i,j} \right|.
	\end{align*}
	We have
	\begin{align*}
	\left| \lambda_1^1 - \lambda_i \right| & \geq \left| \lambda_1 - \lambda_i \right| - \left| \lambda_1 - \lambda_1^1 \right| 
	 \geq \left| \lambda_1 - \lambda_i \right| - \sigma \left|m_{1,1}\right|\\
	& \geq \left| \lambda_1 - \lambda_i \right| - \frac{1}{2} \left| \lambda_1 - \lambda_2 \right|
	\geq \frac{1}{2} \left| \lambda_1 - \lambda_i \right|,
	\end{align*}
	which shows $(ii)$ for $k=0,1$. Thus, for $0<\delta < \alpha$:
	\begin{align*}
	\left| \theta_i^2 - \theta_i^1 \right| &\leq \frac{2 \sigma}{\lambda_1-\lambda_i} C_1 \sqrt{\frac{\log N}{N}} \Delta_1 + \frac{4 \sigma}{\left(\lambda_{1}-\lambda_i\right)^2 } \sigma \left|m_{1,1}\right| \left|m_{i,1}\right|,
	\end{align*} and
	\begin{align*}
	\Delta_2 & \leq  \sigma f_1(N) N^{1/2 + \delta} \Delta_1 + 4 \sigma \left|m_{1,1}\right| \sum_{i=2}^{N} \frac{\sigma \left|m_{i,1}\right|}{\left(\lambda_1 - \lambda_i\right)^2}\\
	& \leq \sigma f_1(N) N^{1/2 + \delta} \Delta_1 + 4 \sigma f_4(N) N^{-1/2} N^{2/3}  \sum_{i=2}^{N} \frac{\sigma \left|m_{i,1}\right|}{\left(\lambda_1 - \lambda_i\right)}\\
	& \leq \sigma f_1(N) N^{1/2 + \delta} \Delta_1 + 4 \sigma f_4(N) N^{1/6}  \Delta_1\\
	& \leq \epsilon \Delta_1.
	\end{align*}
	The proof of $(iii)$ for $k=0,1$ is obvious.\\
	
\textbf{Step 5: Proof of equation $(iv)$.}	
Let $k \geq 2$ and $2 \leq i \leq N$. In the same way as in Step 1, we have
\begin{flalign*}
\left| \theta_i^k - \theta_i^{ k-1} \right| & \leq \frac{2 \sigma C_1}{\lambda_1-\lambda_i} \sqrt{\frac{\log N}{N}} \epsilon^{k-2} \Delta_1 + \frac{8 \sigma^2 C_1^2}{\left(\lambda_1-\lambda_i\right)^2} \frac{\log N}{N} \epsilon^{(k-3)_+}\Delta_1.
\end{flalign*}
In the right-hand term, the ratio of the second term on the first one is smaller that
\begin{equation*}
\frac{4 \sigma C_1}{\lambda_1 - \lambda_i} \sqrt{\frac{\log N}{N}} \epsilon^{-1} \leq \sigma N^{1/6} f(N) \epsilon^{-1} \leq \epsilon \to 0,
\end{equation*}
using Lemma \ref{lemma_controletrousp}, with $f \in \mathcal{F}$. It follows that for $N$ big enough (uniformly in $k$ and $i$) one has
\begin{equation}
\label{ecartik}
\left| \theta_i^k - \theta_i^{ k-1} \right| \leq \frac{\sigma f(N) }{\lambda_1-\lambda_i} N^{-1/2} \epsilon^{k-2} \Delta_1.
\end{equation}
Equation \eqref{ecartik} shows that the scheme \eqref{schemadepicard} converges, and that the limits are indeed the solutions $\theta_1=1, \theta_2, \ldots, \theta_N$ of the fixed-point equations. By a simple summation of (\ref{ecartik}) over $k \geq 2$, applying Lemma \ref{lemma_sommevp1} and inequality \eqref{ineq_encadrementgaussiennes} we have
\begin{flalign*}
\left| \theta_i - \theta_i^{1} \right| & \leq \frac{2 \sigma f(N) }{\lambda_1-\lambda_i} N^{-1/2} \Delta_1 \leq \frac{2 \sigma^2 f(N) }{\lambda_1-\lambda_i} N^{\delta},
\end{flalign*}
where $\delta>0$ is a positive quantity of Lemma \ref{lemma_sommevp1} specified later. Using Lemma \ref{lemma_sommevp2} one has the following control
\begin{equation*}
\sum_{i=2}^{N} \left| \theta_i - \theta_i^{1} \right|^2 \leq 4 \sigma^4 N^{2 \delta} f(N) N^{4/3}.
\end{equation*}
Moreover, Lemma \ref{lemma_concentrationksi} shows that
\begin{equation*}
\sum_{i=2}^{N} \left| \theta_i^{1} \right|^2 \asymp \sigma^2 N^{1/3} \geq g(N)^{-1} \sigma^2 N^{1/3},
\end{equation*}
where $g$ is another function in $\mathcal{F}$. This yields
\begin{equation*}
\sum_{i=2}^{N} \left| \theta_i - \theta_i^{1} \right|^2 \leq \sum_{i=2}^{N} \left| \theta_i^{1} \right|^2 4 \sigma^2 N^{2 \delta+1} {f(N)}{g(N)}.
\end{equation*}
The proof is completed by taking $\delta = \alpha/2$ and applying (\ref{microscopicregime}).
\end{proof}

\subsection*{Proof of Lemma \ref{lemma_controletrousp}}\label{proof_lemma_controletrousp}
\begin{proof}[Proof of Lemma \ref{lemma_controletrousp}]
	This lemma provides a control of the spectral gap $\lambda_1 - \lambda_2$. Given a good rescaling (in $N^{2/3}$), the asymptotic joint law of the eigenvalues in the edge has been investigated in a great amount of research work, for Gaussian ensembles, and for more general Wigner matrices. The GOE case has been mostly studied by Tracy, Widom, and Forrester among many other ; in \cite{Forrester93} and \cite{Tracy98}, the convergence of the joint distribution of the first $k$ eigenvalues towards a density distribution is established:
	\begin{proposition}[\cite{Forrester93}, \cite{Tracy98}]
		For a given $k\geq 1$, and all $s_1, \ldots, s_k$ real numbers,
		\begin{equation}
		\label{loilimitetrousp}
		\mathbb{P}\left(N^{2/3} \left(\lambda_1 - 2\right) \leq s_1, \ldots, N^{2/3} \left(\lambda_k - 2\right) \leq s_k \right) \underset{N \to \infty}{\longrightarrow} \mathcal{F}_{1,k}(s_1, \ldots, s_k),
		\end{equation}
	\end{proposition} where the $\mathcal{F}_{1,k}$ are continuous and can be expressed as solutions of non linear PDEs. Thus the re-scaled spectral gap $N^{2/3}\left(\lambda_{1}-\lambda_2\right)$ has a limit probability density law supported by $\mathbb{R_+}$, which implies that
	\begin{equation*}
	\mathbb{P}\left(N^{2/3} \left(\lambda_1 - \lambda_2\right) \geq \left(\log N\right)^{-C \log \log N} \right) \underset{N \to \infty}{\longrightarrow} 1.
	\end{equation*}
	Of course, the choice of the function $N \mapsto \left(\log N\right)^{-C \log \log N}$ is here arbitrary and the result is also true for any function tending to 0.
\end{proof}

\subsection*{Proof of Lemma \ref{lemma_sommevp2}}\label{proof_lemma_sommevp2}
\begin{proof}[Proof of Lemma \ref{lemma_sommevp2}]
	This result needs an understanding of the behavior of the spectral gaps of matrix $A$, in the bulk and in the edges (left and right). The eigenvalues in the \textit{edge} correspond to indices $i$ such that $i=o(N)$ (left) or $i = N-o(N)$ (right). Eigenvalues in the \textit{bulk} are the remaining eigenvalues. For this, we use a result of rigidity of eigenvalues, due to L. Erdös et al. \cite{Erdos10}, which consists in a control of the probability of the gap between the eigenvales of  $A$ and the typical eigenvalues $\gamma_j$ of the semi-circle law, defined as follows
	\begin{equation}
	\forall i \in \left\lbrace 1, \dots, n\right\rbrace , \; \frac{1}{2 \pi} \int_{-2}^{\gamma_j} \sqrt{4-x^2} dx =1- \frac{j}{N}.
	\end{equation}
	
	\begin{proposition}[\cite{Erdos10}]
		For some positive constants $C_5>0$ and $C_6>0$, for $N$ large enough,
		\begin{multline}
		\label{erdos}
		\mathbb{P}\left( \exists j \in \left\lbrace 1, \dots, n\right\rbrace \, | \, \left|\lambda_j - \gamma_j \right| \geq \left(\log N\right)^{C_5 \log \log N} \left(\min \left(j, N+1-j\right)\right)^{-1/3} N^{-2/3} \right) \\ \leq C_5 \exp \left(- \left(\log N\right)^{C_6 \log \log N} \right).
		\end{multline}
	\end{proposition}	
	\begin{remark}
		Another similar result that goes in the same direction for the GOE is already known: it has been shown by O'Rourke in \cite{ORourke10} that the variables $\lambda_i - \gamma_i$ behave as Gaussian variables when $N \to \infty$. However, the rigidity result in \eqref{erdos} obtained in \cite{Erdos10} can apply in more general models. This quantitative probabilistic statement was not previously known even for the GOE case.
	\end{remark} 
	\begin{remark}
		Let us note that one of the assumptions made in \cite{Erdos10} is that variances of each column sum to 1, which is not directly the case in our model (\ref{GOEmodel}). Nevertheless, one may use (\ref{erdos}) for the re-scaled matrix $\tilde{A} := A \left(1+\frac{1}{N}\right)^{-1/2}$, then easily check that there is a possible step back to $A$: $|\lambda_j - \gamma_j|\leq \left|\lambda_j \left(1+\frac{1}{N}\right)^{-1/2}- \gamma_j \right|+ N^{-1} + o(N^{-1})$, and $ N^{-1} + o(N^{-1}) \leq 2\left(\min \left(j, N+1-j\right)\right)^{-1/3} N^{-2/3}$ for $N$ big enough. Tolerating a slight increase of the constant $C_5$, the result (\ref{erdos}) is thus valid in the GOE.
	\end{remark}
	Let us now compute an asymptotic expansion of $\gamma_j$ in the right edge, which is for  $j = o(N)$. Define
	\begin{equation}
	\label{G}
	G(x) := \frac{1}{2 \pi} \int_{-2}^{x} \sqrt{4-t^2} dt = \frac{x\sqrt{4-x^2}+4 \arcsin(x/2)}{4 \pi}+\frac{1}{2} , 
	\end{equation}
	for all $x \in [-2,2]$. We have $\gamma_j = G^{-1}(1-j/N) = - G^{-1}(j/N)$, observing that the integrand in (\ref{G}) is an even function. We get the following expansion when $x \to -2$,
	\begin{equation*}
	G(x) \underset{x \to -2}{=} \frac{2(x+2)^{3/2}}{3 \pi} + o\left((x+2)^{3/2}\right)
	\end{equation*} which implies that
	\begin{equation*}
	G^{-1}(y) \underset{y \to 0}{=} -2 + \left(\frac{3 \pi y}{2}\right)^{2/3} + o\left(y^{2/3}\right), 
	\end{equation*} hence
	\begin{equation}
	\label{gammaedge}
	\gamma_j \underset{j/N \to 0}{=} 2 - \left(\frac{3 \pi j}{2 N}\right)^{2/3} + o\left((j/N)^{2/3}\right).
	\end{equation} 
	\begin{remark}
		One can observe the coherence of this result that arises naturally in \cite{ORourke10} as the expectation of the eigenvalues in the edge.
	\end{remark}
	
	Let $\epsilon>0$, to be specified later. To establish our result we will split the variables $j$ in three sets:
	\begin{align*}
	A_1 &:= \left\lbrace 2 \leq j \leq \left(\log N\right)^{(C_5+1)\log \log N} \right\rbrace \; \mbox{(a small part of the right edge)},\\ A_2 &:= \left\lbrace \left(\log N\right)^{(C_5+1)\log \log N} < j \leq N^{1-\epsilon} \right\rbrace \;  \mbox{(a larger part of the right edge)}, \\ A_3 &:= \left\lbrace N^{1-\epsilon} < j \leq N \right\rbrace  \; \mbox{(everything else)}.
	\end{align*} We show that the sum over $A_1$ is the major contribution in (\ref{sommevp2eq}). The split in the right edge in $A_1$ and $A_2$ is driven by the error term of (\ref{erdos}): this term is small compared to $\gamma_j$ if and only if $\left(\log N\right)^{C_5 \log \log N} = o(j)$.\\
	
	\textbf{Step 1: estimation of the sum over $A_1$.}
	According to (\ref{erdos}) and Lemma \ref{lemma_controletrousp}, w.h.p.
	\begin{equation*}
	N^{-4/3} \left(\log N\right)^{-C_6 \log \log N}  \leq \left(\lambda_1 - \lambda_2 \right)^2 \leq C_7 N^{-4/3} \left( \log N \right)^{C_6 \log \log N},
	\end{equation*} where $C_6, C_7$ are positive constants. Hence, w.h.p.
	\begin{flalign*}
	\label{a1}
	\frac{N^{4/3}}{C_7 \left( \log N \right)^{C_6 \log \log N}} &\leq \sum_{j \in A_1} \frac{1}{\left(\lambda_1 - \lambda_j\right)^2} \\ & \leq \sum_{j \in A_1} \frac{1}{\left(\lambda_1- \lambda_2\right)^2}\\ & \leq N^{4/3} \left( \log N \right)^{(C_5+C_6+1) \log \log N}. 
	\end{flalign*}
	
	\textbf{Step 2: estimation of the sum over $A_2$.} Let us show that the sum over $A_2$ is asymptotically small compared to the sum over $A_1$: using (\ref{erdos}) and (\ref{gammaedge}), we know that there exists $C_8>0$ such that for all $j \in A_2$, w.h.p.
	\begin{equation*}
	\lambda_j = 2 - C_8 \left(\frac{j}{N}\right)^{2/3} + o\left((j/N)^{2/3}\right),
	\end{equation*} and we know furthermore (se e.g. \cite{Anderson09}) that w.h.p.
	\begin{equation}
	\label{lambda1}
	\lambda_1 = 2 + o\left((j/N)^{2/3}\right), \forall j \in A_2
	\end{equation} hence w.h.p.
	\begin{flalign*}
	\label{a2}
	\sum_{j \in A_2} \frac{1}{\left(\lambda_1 - \lambda_j\right)^2} &= N^{4/3} \sum_{j \in A_2} \frac{1}{C_9 j^{4/3}(1+o(1))}\\& =  N^{4/3} (1+o(1)) \sum_{j \in A_2} \frac{1}{C_9 j^{4/3}} = o\left(N^{4/3}\right),
	\end{flalign*}
	using in the last line the fact that the Riemann's series $\sum j^{-4/3}$ converges.\\
	
	\textbf{Step 3: estimation of the sum under $A_3$.} With the previous results (\ref{erdos}), (\ref{gammaedge}) and (\ref{lambda1}), assuming that $\epsilon<1$, we get w.h.p.
	\begin{equation*}
	\lambda_1 - \lambda_{N^{1-\epsilon}} = C_8 N^{-2\epsilon/3} + O\left(N^{-2\epsilon/3}\right),
	\end{equation*} which gives w.h.p. the following control
	\begin{flalign*}
	\label{a3}
	\sum_{j \in A_3} \frac{1}{\left(\lambda_1- \lambda_j\right)^2} &\leq \left(N-N^{1-\epsilon}\right) \frac{1}{\left(\lambda_1 - \lambda_{N^{1-\epsilon}}\right)^2}\\& = \left(N-N^{1-\epsilon}\right) \frac{N^{4\epsilon/3}}{C_9 (1+o(1))} = O\left(N^{1+4\epsilon/3}\right) = o\left(N^{4/3}\right),&&
	\end{flalign*}
	as long as $\epsilon < 1/4$. Taking such a $\epsilon$, these three controls end the proof.
\end{proof}

\subsection*{Proof of Lemma \ref{lemma_sommevp1}} \label{proof_lemma_sommevp1}
\begin{proof}[Proof of Lemma \ref{lemma_sommevp1}]
	We follow the same steps as in the proof of Lemma \ref{lemma_sommevp2}. Let's take $\delta>0$. We split the $j$ variables in three sets:
	\begin{align*}
	A_1 &:= \left\lbrace 2 \leq j \leq N^{1/3} \right\rbrace,\\ A_2 &:=  \left\lbrace N^{1/3} < j \leq N^{1-\delta} \right\rbrace,\\ A_3 &:= \left\lbrace N^{1-\delta} < j \leq N \right\rbrace.
	\end{align*}
	We use Lemma \ref{lemma_controletrousp} to obtain the following control w.h.p.
	\begin{flalign*}
	\sum_{j \in A_1} \frac{1}{\lambda_1 - \lambda_j} \leq N^{1/3} N^{2/3} \left(\log N\right)^{C_5 \log \log N} = O(N^{1+\delta}).
	\end{flalign*}
	Similarly, for $A_2$
	\begin{flalign*}
	\sum_{j \in A_2} \frac{1}{\lambda_1 - \lambda_j} &\leq \sum_{j \in A_2} \frac{1}{o(N^{-2/3}) + C_8(j/N)^{2/3} + O\left(\left(\log N\right)^{C_5 \log \log N}N^{-2/3}j^{-1/3}\right)} \\ &= N^{2/3} \sum_{j \in A_2} \frac{1}{o(1)+C_8j^{2/3}} \leq C_{10} N^{2/3} N^{(1-\delta)/3} \leq O(N^{1+\delta}).
	\end{flalign*}
	Finally, using Cauchy–Schwarz inequality
	\begin{flalign*}
	\sum_{j \in A_3} \frac{1}{\lambda_1 - \lambda_j} \leq  \sqrt{N} \left(\sum_{j \in A_3} \frac{1}{\left(\lambda_1 - \lambda_j\right)^2}\right)^{1/2} \leq  \sqrt{N} O(N^{1/2+2\delta/3}) = O(N^{1+\delta}).
	\end{flalign*}
\end{proof}

\subsection{Proof of Lemma \ref{lemma_concentrationksi}} \label{proof_lemma_concentrationksi}
\begin{proof}[Proof of Lemma \ref{lemma_concentrationksi}]
We show that w.h.p.
\begin{equation}
\label{concentration}
\sum_{i=2}^{N} \frac{\langle Hv_i,v_1 \rangle ^2 }{\left(\lambda_1 - \lambda_i\right)^2} - \frac{1}{N} \sum_{i=2}^{N} \frac{1}{\left(\lambda_1 - \lambda_i\right)^2}
= o\left(\frac{1}{N} \sum_{i=2}^{N} \frac{1}{\left(\lambda_1 - \lambda_i\right)^2}\right)
\end{equation}
Let us recall that $H$ is drawn according to the GOE, hence its law is invariant by rotation. This implies that the $\langle Hv_i,v_1 \rangle$ are independent variables with variance $1/N$, independent of $\lambda_1, \ldots, \lambda_N$. Define
\begin{equation*}
M_N := \sum_{i=2}^{N} \frac{\langle Hv_i,v_1 \rangle ^2 - 1/N }{\left(\lambda_1 - \lambda_i\right)^2}.
\end{equation*}
Computing the second moment of $M_N$, we get
\begin{flalign*}
\mathbb{E}\left[M_N ^2 | \lambda_1, \ldots, \lambda_N \right]& = \mathrm{Var}(M_N | \lambda_1, \ldots, \lambda_N) = \frac{1}{N^4} \sum_{i=2}^{N} \frac{2 }{\left(\lambda_1 - \lambda_i\right)^4}. 
\end{flalign*}
Adapting the proof of Lemma \ref{lemma_sommevp2}, following the same steps, one can also show that w.h.p.
\begin{equation}
\label{controlevp4}
\sum_{i=2}^{N} \frac{1}{\left(\lambda_1 - \lambda_i\right)^4} \asymp N^{8/3}.
\end{equation}
Let $\epsilon = \epsilon(N)>0$ to be specified later. By Markov's inequality
\begin{flalign*}
\mathbb{P}\left(\left|M_N\right| \geq \frac{\epsilon}{N} \sum_{i=2}^{N} \frac{1}{\left(\lambda_1 - \lambda_i\right)^2} | \lambda_1, \ldots, \lambda_N \right)& \leq \frac{N^2}{\epsilon^2} \frac{\mathbb{E}\left[M_N ^2 | \lambda_1, \ldots, \lambda_N \right]}{\left(\sum_{i=2}^{N} \frac{1}{\left(\lambda_1 - \lambda_i\right)^2}\right)^2} \\
&  \asymp \frac{1}{\epsilon^2 N^2},
\end{flalign*}
by Lemma \ref{lemma_sommevp2} and equation (\ref{controlevp4}). Taking e.g. $\epsilon(N)=N^{-1/2}$ concludes the proof.
\end{proof}

\section{Additionnal proofs for Sections \ref{toymodel} \& \ref{EIG1threshold}}
\label{section45_add_proofs}
\subsection{Proof of Lemma \ref{fibo}}\label{proof_lemma_fibo}
\begin{proof}[Proof of Lemma \ref{fibo}]
We fix $\alpha>0$ and we want to prove
\begin{equation}
\label{sum_fibop}
\sum_{k=0}^{\lfloor(N-1)/2\rfloor}  \binom{N-1-k}{k} \alpha^k = \frac{1}{\sqrt{1+4 \alpha}} \left[\left(\frac{1+\sqrt{1+4 \alpha}}{2}\right)^N - \left(\frac{1-\sqrt{1+4 \alpha}}{2}\right)^N\right].
\end{equation}We denote in the following $\phi_+ := \frac{1+\sqrt{1+4 \alpha}}{2}$ and $\phi_- := \frac{1-\sqrt{1+4 \alpha}}{2}$, and for all $N \geq 1$:
\begin{equation*}
u_N = u_N(\alpha) := \sum_{k=0}^{\lfloor(N-1)/2\rfloor}  \binom{N-1-k}{k} \alpha^k.
\end{equation*}
We clearly have $u_N(\alpha) \leq \left(1+\alpha\right)^N$. For all $t>0$ small enough (e.g. $t < \frac{1}{1+\alpha}$), define
\begin{equation*}
f(t) := \sum_{N =1}^{\infty}  u_N t^N.
\end{equation*}
On one hand,
\begin{flalign*}
\frac{t}{1-t-\alpha t^2} & = t \sum_{m=0}^{\infty} (t+\alpha t^2)^m = \sum_{m=0}^{\infty} \sum_{l=0}^{m} \binom{m}{l} \alpha ^l t^{l + m + 1}\\
& = \sum_{N=1}^{\infty} \left(\sum_{\substack{0 \leq l \leq m \\ l+m = N-1}} \binom{m}{l} \alpha^l \right) t^N = \sum_{N=1}^{\infty} u_N t^N = f(t). 
\end{flalign*}
On the other hand,
\begin{flalign*}
\frac{t}{1-t-\alpha t^2} & = \frac{t}{\left(1-\phi_{-}t\right)\left(1-\phi_{+}t\right)} =  \frac{1}{\phi_{+}-\phi_{-}}\left(\frac{1}{1-\phi_{+}t}-\frac{1}{1-\phi_{-}t}\right)\\
& =  \frac{1}{\sqrt{1+4\alpha}} \sum_{N=1}^{\infty}\left(\phi_{+}^N - \phi_{-}^N\right) t^N. 
\end{flalign*}
This proves $(\ref{sum_fibop})$.
\end{proof}

\subsection{Proof of Lemma \ref{stillthelink}}
\begin{proof}[Proof of Lemma \ref{stillthelink}]\label{proof_stillthelink}
Let us represent the situation in the plane spanned by $v_1$ and $v'_1$, as shown on figure \ref{imagev}.
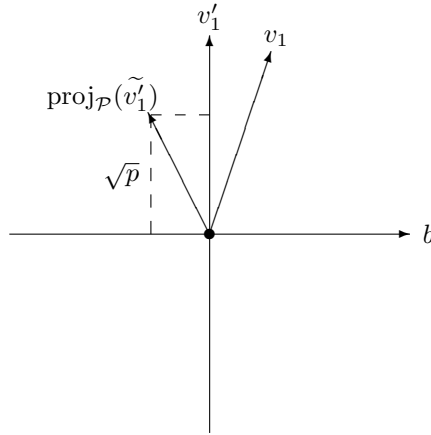
\begin{figure}[H]
	\centering
	\begin{picture}(150,150)

\put(75,75){\circle*{4}}
\put(0,75){\vector(1,0){150}}
\put(157,75){\makebox(0,0){$b$}}
\put(75,0){\vector(0,1){150}}
\put(75,157){\makebox(0,0){$v'_1$}}

\put(75,75){\vector(1,3){23}}
\put(100,149){\makebox(0,0){$v_1$}}

\put(75,75){\vector(-1,2){23}}
\put(35,128){\makebox(0,0){$\mathrm{proj}_{\mathcal{P}}(\widetilde{v'_1})$}}

\multiput(53,75)(0,10){5}{\line(0,1){5}}
\multiput(75,120)(-9,0){3}{\line(-1,0){4}}

\put(42,97){\makebox(0,0){$\sqrt{p}$}}

\end{picture}
	\caption{\label{imagev} Orthogonal projection of $\widetilde{v}_1$ on $\mathcal{P} := \mathrm{span}(v'_1, v_1)$.}
\end{figure}
Since $\widetilde{v}_1$ is taken such that $\langle v_1,\widetilde{v}_1 \rangle >0$ and $\sigma_1$ satisfies (\ref{microscopicregime}), we have  $\langle \widetilde{v}_1,v'_1 \rangle >0$ for $N$ large enough by Proposition \ref{prop_gaussian_decomp}. Let $p := \langle \widetilde{v}_1,v'_1 \rangle^2$ and $\widetilde{w} := \widetilde{v}_1 - \sqrt{p} v'_1 \in \left(v'_1\right)^{\perp}$. By invariance by rotation we can obtain that $\frac{\widetilde{w}}{\| \widetilde{w}\|}=\frac{\widetilde{w}}{\sqrt{1-p}}$ is uniformly distributed on the unit sphere $\mathbb{S}^{N-2}$ of $\left(v'_1\right)^{\perp}$, and independent of $p, v_1$ and $v'_1$. Hence 
\begin{equation*}
\langle b,\widetilde{v}_1 \rangle = \langle b, \widetilde{w} \rangle \overset{(d)}{=} \sqrt{1-p} \cdot \frac{\widetilde{Z}_1}{\sqrt{\sum_{i=1}^{N-1} \left(\widetilde{Z}_i\right)^2}},
\end{equation*} where the $\widetilde{Z}_i$ are independent Gaussian standard variables, independent from everything else.
According to Section \ref{linkGOEtoy} we know that $1 - \langle v_1,v'_1 \rangle \asymp \sigma_1^2 N^{1/3} $ and thus $\langle v_1,b \rangle \asymp \sigma_1 N^{1/6}$. This yields, for $N$ large enough, w.h.p,
\begin{flalign*}
0 < \langle \widetilde{v}_1, v_1 \rangle & \leq \sqrt{p} \langle v_1, v'_1 \rangle + \sqrt{\frac{1-p}{N}} \widetilde{Z}_1 \sigma_1 N^{1/6} f(N)\\
& \leq \sqrt{p} \langle v_1, v'_1 \rangle + \sqrt{1-p} N^{-4/3} g(N) \\
& \leq \max \left(\sqrt{p},\sqrt{1-p}\right) \langle v_1, v'_1 \rangle \\
& \leq \langle v_1, v'_1 \rangle,
\end{flalign*} where $f$ and $g$ are two functions as defined in Lemma \ref{lemma_sommevp2}. From this point one can still make the link with the toy model, as done in the beginning of section \ref{toymodel}. By invariance by rotation, letting $t := \widetilde{v}_1 - \langle \widetilde{v}_1,v_1 \rangle v_1$, we know that $\|t\|$ and $\frac{t}{\| t\|}$ are independent, and that $\frac{t}{\| t\|}$ is uniformly distributed on the unit sphere in $v_1^{\perp}$. We have the following equality in distribution:
\begin{equation*}
\left(r_1(v_1),r_1(\widetilde{v}_1)\right) \overset{(d)}{=} \left(r_1(X),r_1(X+\mathbf{s}Z) \right),
\end{equation*} 
with w.h.p. $$\mathbf{s} \geq \mathbf{s^1} = \frac{\|w\| \|X\|}{\left(\sum_{i=2}^{N} Z_i^2\right)^{1/2} \left(1 - \frac{\|w\| Z_1}{\left(\sum_{i=2}^{N} Z_i^2\right)^{1/2}}\right)} \asymp \sigma_1 N^{1/6},$$ where the $X_i$, $Z_i$ and $w$ are defined in section \ref{toymodel}, for $\sigma=\sigma_1$.
\end{proof}

\subsection{Proof of Lemma \ref{Pi+casei}} \label{proof_lemma_Pi+(i)}
\begin{proof}[Proof of Lemma \ref{Pi+casei}]
Recall that we work in the case $(i)$ ($\sigma = o(N^{-7/6-\epsilon})$ for some $\epsilon>0$), with $\langle v_1,v'_1 \rangle >0$ and $\Pi = \mathrm{Id}$. We want to show that w.h.p.
\begin{equation}
\label{Pi+caseieq}
	\langle A, \Pi_{+} B \Pi_{+}^T \rangle > \langle A, \Pi_{-} B \Pi_{-}^T \rangle.
\end{equation} 
Define
\begin{equation*}
\mathcal{G} := \left\lbrace i, \Pi_{+}(i)=\Pi(i)=i \right\rbrace.
\end{equation*} and
\begin{equation*}
\mathcal{A} := \left\lbrace \sigma N^{1/6}f(N)^{-1} \leq \mathbf{s} \leq \sigma N^{1/6 }f(N) \right\rbrace,
\end{equation*} with $f \in \mathcal{F}$ such that $\mathbb{P} \left(\mathcal{A}\right) \to 1$. For $N$ large enough, on the event $\mathcal{A}$, we have $ 0 \leq \mathbf{s}N \leq  N^{-\epsilon }f(N)$. Hence, retaking the proof of Proposition \ref{prop_zero_one_toy}, we have 
\begin{flalign*}
\phi_{x,z}\left(N, \mathbf{s} \right) & \geq \mathbb{P} \left(\mathcal{N}^{+}(x,x+\mathbf{s}z)=\mathcal{N}^{-}(x,x+\mathbf{s}z)=0\right) \\
& \sim \exp\left(- \mathbf{s}NE(x)\left[z(2F(z)-1)+2E(z)\right]\right) = 1 - O(N^{-\epsilon}f(N)).&&
\end{flalign*}

Thus, with dominated convergence, for $N$ large enough,
\begin{flalign}
\label{probassurA}
\mathbb{P}\left({\Pi_{+}}(i)=\Pi(i) | \mathcal{A}  \right)  = \iint dx dz E(x) E(z) \mathbb{E}\left[\phi_{x,z}\left(N, \mathbf{s} \right)| \mathcal{A}\right] \geq 1 - O(N^{-\epsilon}f(N)).
\end{flalign} We use Markov's inequality with (\ref{probassurA}) to show that $\mathbb{P}\left(\sharp \mathcal{G} \leq N- N^{1-\epsilon/2} \; | \; \mathcal{A} \right) \leq O\left(N^{-\epsilon/2}f(N)\right) $, hence w.h.p.
\begin{equation}
\label{controlgoods}
\sharp \mathcal{G} \geq N- N^{1-\epsilon/2}.
\end{equation}

Splitting the sum
\begin{flalign*}
\langle A, \Pi_{+} B \Pi_{+}^T \rangle & = \sum_{i,j} A_{i,j} B_{\Pi_{+}(i),\Pi_{+}(j)} = \sum_{(i,j) \in \mathcal{G}^2 } A_{i,j} B_{i,j} + \sum_{(i,j) \notin \mathcal{G}^2 } A_{i,j} B_{\Pi_{+}(i),\Pi_{+}(j)},
\end{flalign*} one has, w.h.p.,
\begin{multline*}
\langle A, \Pi_{+} B \Pi_{+}^T \rangle = \sum_{(i,j) \in \mathcal{G}^2 } A_{i,j}^2 + \sum_{\substack{(i,j) \notin \mathcal{G}^2\\ (\Pi_{+}(i),\Pi_{+}(j)) \neq (j,i) } } A_{i,j} A_{\Pi_{+}(i),\Pi_{+}(j)} \\+ \sum_{\substack{(i,j) \notin \mathcal{G}^2\\ (\Pi_{+}(i),\Pi_{+}(j)) = (j,i) } } A_{i,j}^2 + \sigma \sum_{1 \leq i,j \leq N } A_{i,j} H_{\Pi_{+}(i),\Pi_{+}(j)} \\ 
\geq C_1 \frac{(\sharp \mathcal{G})^2}{N} - C_2 \left(N^2-(\sharp \mathcal{G})^2\right)\frac{\log N}{N} -  C_2 \sigma N^2 \frac{\log N}{N}.
\end{multline*} We applied the law of large numbers for the first sum, lower-bounded the third sum by zero, and the classical inequality $\max_{i,j}\left\lbrace A_{i,j}, H_{i,j} \right\rbrace \leq C_2 \frac{\log N}{N}$ (which holds w.h.p.) for the two others. \\
Inequality (\ref{controlgoods}) and condition $(i)$ lead to, w.h.p.
\begin{equation*}
\langle A, \Pi_{+} B \Pi_{+}^T \rangle \geq C_1 N - 2 C_1 N^{1- \epsilon/2} - 2 C_2 N^{1- \epsilon/2} \log N - C_2 N^{-1/6-\epsilon}\log N \geq C_3 N.
\end{equation*}
On the other hand, since by definition $\Pi_{-}(i)=\Pi_{+}(N+1-i)$, w.h.p.,
\begin{multline*}
\langle A, \Pi_{-} B \Pi_{-}^T \rangle = \sum_{(i,j) \in \mathcal{G}^2 } A_{i,j}B_{N+1-i,N+1-j} + \sum_{\substack{(i,j) \notin \mathcal{G}^2 } } A_{i,j} B_{\Pi_{-}(i),\Pi_{-}(j)} \\
\leq O(\log N) + \frac{(\sharp \mathcal{G})^2 }{N}o(1) + C_2 \left(N^2-(\sharp \mathcal{G})^2\right)\frac{\log N}{N}.
\end{multline*} For the first sum, we used the law of large numbers: the variables $A_{i,j}$ and $B_{N+1-i,N+1-j}$ are independent in all cases but at most $N+1$, and this part of the sum is bounded by $O(\log N) $. We used the same control on Gaussian variables as above. \\
This gives
\begin{equation*}
\left(\langle A, \Pi_{-} B \Pi_{-}^T \rangle \right)_{+} = o_{\mathbb{P}}(N),
\end{equation*} where $(x)_{+} := \max (0,x)$, which proves (\ref{Pi+caseieq}).
\end{proof}

\subsection{Proof of Lemma \ref{Pi+caseii}}
\begin{proof}[Proof of Lemma \ref{Pi+caseii}]\label{proof_lemma_Pi+caseii}
Recall that we work in the case $(ii)$ ($\sigma = \omega(N^{-7/6+\epsilon})$ for some $\epsilon>0$), with $\langle v_1,v'_1 \rangle >0$ and $\Pi = \mathrm{Id}$. We want to show that the aligning permutation between $v_1$ and $-v'_1$ has a very bad overlap. Taking the couple $(X,-Y)$ where $(X,Y) \sim \mathcal{J}(N,s)$, one can adapt the proof of Proposition \ref{prop_zero_one_toy}, with the new definitions
\begin{align*}
	\widetilde{S^{+}}(x,y) &:= \mathbb{P}\left(X_1 > x, -Y_1 < -y\right), \mbox{ and}\\
	\widetilde{S^{-}}(x,y) &:= \mathbb{P}\left(X_1 < x, -Y_1 > -y\right).
\end{align*} The analysis is even easier since for all $x,z$, there exist two constants $c,C$ such that $$0 < c \leq \widetilde{S^{+}}(x,x+s z) , \, \widetilde{S^{-}}(x,x+ sz) \leq C <1.$$
It is then easy to check that the proof of Proposition \ref{prop_zero_one_toy}, case $(ii)$ adapts well.
\end{proof}
 
\newpage

\section*{Acknowledgments}
This work was partially supported by the French government under management of Agence Nationale de la Recherche as part of the “Investissements d’avenir” program, reference ANR19-P3IA-0001 (PRAIRIE 3IA Institute).

\bibliographystyle{plain}
\bibliography{biblio_EIG1}

\end{document}